\documentclass[12pt,a4paper]{amsart}
\usepackage{t1enc}
\usepackage[utf8]{inputenc}
\usepackage{amssymb,enumitem,mathrsfs}
\usepackage[margin=1in]{geometry}
\usepackage{hyperref}

\newcommand{\gener}{\mathcal{A}}

\newcommand{\ind}{\mathbf{1}}

\numberwithin{equation}{section}

\newtheorem{thm}{Theorem}[section]
\newtheorem{lem}[thm]{Lemma}
\newtheorem{prop}[thm]{Proposition}
\newtheorem{cor}[thm]{Corollary}
\theoremstyle{definition}

\newtheorem{exmp}[thm]{Example}
\newtheorem{remark}[thm]{Remark}

\newcommand{\sub}{\subseteq}
\newcommand{\R}{\mathbb{R}}

\renewcommand{\leq}{\leqslant}
\renewcommand{\le}{\leq}
\renewcommand{\geq}{\geqslant}
\renewcommand{\ge}{\geq}

\DeclareMathOperator{\Id}{Id}
\DeclareMathOperator{\dist}{dist}

\DeclareMathOperator{\re}{Re}

\newcommand{\ex}{\mathbb{E}}
\newcommand{\pr}{\mathbb{P}}

\newcommand{\cbarp}{C_{\mathrm{reg}}}

\newcommand{\cbhi}{C_{\mathrm{BHI}}}
\newcommand{\cbhit}{\tilde{C}_{\mathrm{BHI}}}
\newcommand{\cnu}{C_{\text{\rm L\'evy}}}
\newcommand{\cnus}{\tilde{C}_{\text{\rm L\'evy}}}

\newcommand{\cro}{\varrho}
\newcommand{\cgreen}{C_{\mathrm{green}}}
\newcommand{\ctau}{C_{\mathrm{exit}}}

\begin{document}

\title[Potential kernels, probabilities of hitting a ball, harmonic functions\ldots]{Potential kernels, probabilities of hitting a ball, harmonic~functions and the boundary Harnack inequality for unimodal L\'evy processes}

\author{Tomasz Grzywny and Mateusz Kwa\'{s}nicki}
\thanks{Tomasz Grzywny was supported in part by National Science Centre (Poland) grant no.\@ 2014/14/M/ST1/00600 and the Alexander von Humboldt Foundation. Mateusz Kwa\'{s}nicki was supported by National Science Centre (Poland) grant no.\@ 2011/03/D/ST1/00311.}
\address{Faculty of Pure and Applied Mathematics \\ Wroc{\l}aw University of Science and Technology \\ ul.\@ Wybrze{\.z}e Wyspia{\'n}skiego 27, 50-370 Wroc{\l}aw, Poland}
\begin{abstract}
%
In the first part of this article, we prove two-sided estimates of hitting probabilities of balls, the potential kernel and the Green function for a ball for general isotropic unimodal L\'evy processes. Our bounds are sharp under the absence of the Gaussian component and a mild regularity condition on the density of the L\'{e}vy measure: its radial profile needs to satisfy a scaling-type condition, which is equivalent to $O$-regular variation at zero and at infinity with lower indices greater than $-d - 2$. We also prove a supremum estimate and a regularity result for functions harmonic with respect to a general isotropic unimodal L\'evy process.

In the second part we apply the recent results on the boundary Harnack inequality and Martin representation of harmonic functions for the class of isotropic unimodal L\'evy processes characterised by a localised version of the scaling-type condition mentioned above. As a sample application, we provide sharp two-sided estimates of the Green function of a half-space.

Our results are expressed in terms of Pruitt's functions $K(r)$ and $L(r)$, measuring local activity and the amount of large jumps of the L\'evy process, respectively.
\end{abstract}

\maketitle

%
%

\section{Introduction}\label{sec:i}

The present article further develops the potential theory of isotropic unimodal L\'evy process in $\R^d$, extending previous works in this area (\cite{BGR,BGR2,BGR1,CGT,TG,MR2928720,MR3206464,MR2513121,MR2928332,MR2994122,MR3131293}). Our main results are: estimates of hitting probabilities of balls, estimates of the Green function of the full space (the potential kernel), half-spaces and balls; supremum estimate and regularity result for harmonic functions; and the boundary Harnack inequality. Many of our results hold for arbitrary isotropic unimodal L\'{e}vy processes. For the boundary Harnack inequality and boundary estimates for a half-space, we need to assume some regularity of jumps, and that the process has no Gaussian component.

Throughout the article we assume that $X_t$ is a (non-constant) isotropic unimodal L\'evy process with values in $\R^d$. By $\sigma^2$ we denote its Gaussian coefficient. The L\'evy measure of $X_t$ is an isotropic unimodal measure: it has a radial density function $\nu(z)$, and the radial profile of $\nu$ (denoted by the same symbol $\nu$) is non-increasing. Following Pruitt~(\cite{MR632968}), we let
\begin{equation}\label{eq:KL}
\begin{aligned}
 K(r) & = \frac{\sigma^2 d}{r^2} + \int_{B(0, r)} \frac{|z|^2}{r^2} \nu(z) dz , \qquad & L(r) & = \int_{\R^d \setminus B(0, r)} \nu(z) dz
\end{aligned}
\end{equation}
for $r > 0$. For a formal introduction of these and other related objects (including regular harmonic functions, the potential kernel, and the Green function), see Section~\ref{sec:pre}.

Our first three results provide estimates of the probability of hitting a ball, the potential kernel and the Green function of a ball. All of them improve previously known results: for hitting probabilities, see Lemma~2.5 in~\cite{MR3148015}, Proposition~5.8 in~\cite{BGR} and Lemmas~3.4 and~3.5 in~\cite{KX}; for the potential kernel, we refer to Theorem~3 in~\cite{TG}, Theorem~5.8 in~\cite{GRT}, Proposition~4.5 in~\cite{MR2928720}, and Theorem~3.2 in~\cite{MR2986850}.
 
\begin{thm}\label{thm:ret}
Suppose that $d \ge 3$, and let $X_t$ be an isotropic unimodal L\'evy process in~$\R^d$. For $r > 0$ denote by $T_r$ the hitting time of a closed ball $\overline{B}(0, r)$. Then
\begin{equation}\label{eq:ret:u1}
 \pr^x(T_r < \infty) \le c(d) \, \frac{K(|x|)}{K(|x|) + L(|x|)} \, \frac{r^d (K(r) + L(r))}{|x|^d (K(|x|) + L(|x|))}
\end{equation}
when $|x| \ge 2 r$. Furthermore, in any dimension $d \ge 1$,
\begin{equation}\label{eq:ret:l}
 \pr^x(T_r < \infty) \ge \frac{1}{c(d)} \, \frac{|x|^d \nu(x)}{K(|x|) + L(|x|)} \, \frac{r^d (K(r) + L(r))}{|x|^d (K(|x|) + L(|x|))}
\end{equation}
when $|x| \ge r$
\end{thm}

\begin{thm}\label{thm:pot}
Suppose that $d \ge 3$, and let $X_t$ be an isotropic unimodal L\'evy process in~$\R^d$. Denote by $U(x)$ the potential kernel of $X_t$. Then
\begin{equation}\label{eq:pot:u}
 U(x) \le c(d) \, \frac{K(|x|)}{|x|^d (K(|x|) + L(|x|))^2}
\end{equation}
for all $x \ne 0$. Furthermore, in any dimension $d \ge 1$,
\begin{equation}\label{eq:pot:l}
 U(x) \ge \frac{1}{c(d)} \, \frac{\nu(x)}{(K(|x|) + L(|x|))^2}
\end{equation}
for all $x \ne 0$.
\end{thm}

\begin{thm}\label{thm:green}
Let $X_t$ be an isotropic unimodal L\'evy process in~$\R^d$. Let $G_{B(0, r)}(x, y)$ denote the Green function of a ball $B(0, r)$. Then
\begin{equation}\label{eq:green:u}
 G_{B(0, r)}(x, y) \le c(d) \, \frac{K(|x - y|)}{|x - y|^d (K(r) + L(r))^2}
\end{equation}
for all $r > 0$ and $x, y \in B(0, r)$ such that $x \ne y$. Furthermore, if $r_x = \min(|x - y|, r - |x|)$ and $r_y = \min(|x - y|, r - |y|)$, then
\begin{equation}\label{eq:green:l}
 G_{B(0, r)}(x, y) \ge \frac{1}{c(d)} \, \frac{\nu(x - y)}{(K(r_x) + L(r_x)) (K(r_y) + L(r_y))}
\end{equation}
for all $r > 0$ and all $x, y \in B(0, r)$ such that $x \ne y$.
\end{thm}

\begin{remark}\label{rem:pot}
\begin{enumerate}[label=(\alph*)]
\item If $d \ge 3$, then the upper bound of Theorem~\ref{thm:green} is worse than the bound obtained from Theorem~\ref{thm:pot} and the inequality $G_{B(0, r)}(x, y) \le U(x - y)$. However, if $d \le 2$, $U(x - y)$ may be infinite, while the upper bound of Theorem~\ref{thm:pot} is non-trivial.
\item The lower bound of Theorem~\ref{thm:pot} is a direct consequence of the lower bound of Theorem~\ref{thm:green}.
\item Our results hold also for compound Poisson processes. When the total mass of the L\'evy measure is $1$, then the probability of hitting a ball, the potential kernel and the Green function for a ball are the same whether defined for the continuous-time process $X_t$ or for the underlying (discrete-time) random walk. Thus, the results of Theorems~\ref{thm:ret}, \ref{thm:pot} and~\ref{thm:green} extend to isotropic unimodal random walks. As far as we know, these results are new in this context. (Note that for compound Poisson processes, the potential kernel and the Green function have an atom at $y = x$).
\item\label{rem:pot:orv} Clearly, $r^d \nu(r) \le c(d) K(r)$. As discussed in Remark~\ref{rem:orv} below and in Appendix~\ref{sec:app:orv}, the opposite inequality, $r^d \nu(r) \ge C K(r)$ (for some $C > 0$ and all $r > 0$), is equivalent to $O$-regular variation of $\nu$ at zero and at infinity, with lower indices greater than $-d - 2$. In this case the upper and lower bounds in Theorems~\ref{thm:ret} and~\ref{thm:pot} are comparable for all $r > 0$; the same is true for the lower and upper bounds of Theorem~\ref{thm:green}, as long as $|x - y|$, $r - |x|$ and $r - |y|$ are comparable with~$r$.
\item The lower bounds in Theorems~\ref{thm:ret}, \ref{thm:pot} and~\ref{thm:green} are clearly not sharp whenever $\nu(x) = 0$ (or $\nu(x - y) = 0$), and thus, in particular, if $X_t$ is the Brownian motion.
\item\label{rem:pot:u} Also the upper bounds in Theorems~\ref{thm:ret}, \ref{thm:pot} and~\ref{thm:green} are not sharp; a counterexample here is, however, more involved: it is an isotropic unimodal L\'evy process with $\sigma^2 = 1$ and with $\nu(z) = \lambda |B_R|^{-1} \ind_{B_R}(z)$, where $R$ and $\lambda$ are very large, but at the same time $R^{-d} \lambda$ is very close to $0$. See Example~\ref{ex:sup} for a detailed description of a slightly different example.
\item Although not always sharp, the upper bounds are not far away from optimal ones. For example, by integrating both sides of upper bound~\eqref{eq:pot:u} of Theorem~\ref{thm:pot} over $B(0, r)$ and using the identity $(K + L)'(r) = -2 K(r) / r$, one can recover the sharp estimate $\int_{B(0, r)} U(z) dz \le c(d) / (K(r) + L(r))$ for the potential measure of a ball. The latter one is proved (with a different method) in Proposition~2 in~\cite{TG}; see also~\eqref{eq:pot:bound} below.
\end{enumerate}
\end{remark}

Our next theorem provides a supremum estimate for harmonic functions: the $L^\infty$ norm of a function inside the domain of harmonicity (and away from the boundary) is controlled by the $L^1$ norm in full space. We state the result for non-negative functions only; the extension to signed functions is immediate.

A similar supremum estimate was found for rather general Markov processes in~\cite{BKK}, using an elaborate probabilistic argument. However, when restricted to isotropic unimodal L\'evy processes, it requires some extra regularity assumptions. Analogous statements for other classes of operators are known in the context of non-local partial differential equations, see, for example, Theorem~5.1 in~\cite{MR2831115}.

\begin{thm}
\label{thm:sup}
Let $X_t$ be an isotropic unimodal L\'evy process. Suppose that $0 \le q < r$ and that a non-negative function $f$ is a regular harmonic function in $B(x_0, r)$ (with respect to $X_t$). Then
\begin{equation}\label{eq:sup:u}
 f(x) \le c(d, q/r) \, \frac{K(r)}{r^d (K(r) + L(r))} \int_{\R^d \setminus B(x_0, q)} f(z) dz
\end{equation}
for all $x \in \overline{B}(x_0, q)$. Furthermore,
\begin{equation}\label{eq:sup:l}
 f(x_0) \ge \frac{1}{c(d, q/r)} \, \frac{\nu(r)}{K(r) + L(r)} \int_{B(x_0, r) \setminus B(x_0, q)} f(z) dz
\end{equation}
\end{thm}

\begin{remark}
\label{rem:sup}
\begin{enumerate}[label=(\alph*)]
\item A more refined statement is given in Proposition~\ref{prop:sup}.
\item The proof of our result is rather elementary, at least when compared to that of~\cite{BKK}. It resembles the volume averaging argument used in~\cite{Bogdan97,BKK08}.
\item\label{rem:sup:orv}
As in Remark~\ref{rem:pot}\ref{rem:pot:orv}, the upper and lower bounds in Theorem~\ref{thm:sup} are comparable for all $r > 0$ (provided that the ratio $q / r$ is fixed) if and only if $\nu$ is $O$-regularly varying at zero and at infinity, with lower indices greater than $-d - 2$ (see Remark~\ref{rem:orv} and Appendix~\ref{sec:app:orv}).
\item The lower bound~\eqref{eq:sup:l} is clearly not sharp whenever $\nu(r) = 0$, and thus, in particular, if $X_t$ is the Brownian motion. In some cases a better constant in the lower bound can be found by a more careful applications of Proposition~\ref{prop:sup}.
\item Also the upper bound~\eqref{eq:sup:u} is not sharp. Clearly, the non-local term (the integral over $\R^d \setminus B(x_0, r)$) is not needed when $X_t$ is the Brownian motion, but also the constant in the local term (the integral over $B(x_0, r)$) may fail to be optimal. A counterexample here is the same as in Remark~\ref{rem:pot}\ref{rem:pot:u}, see also Example~\ref{ex:sup}.
\end{enumerate}
\end{remark}

The supremum estimate typically leads to regularity of harmonic functions. Our next theorem provides a sample development of this kind. Common results in that direction assert H\"{o}lder regularity of harmonic functions, see~\cite{MR3161529,TG,MR2572166,MR3065312,MR2031452}. On the other hand, for some class of smooth kernels it is known that harmonic functions are smooth, see~\cite{MR3276603}. 

\begin{thm}
\label{thm:reg}
Let $X_t$ be an isotropic unimodal L\'evy process with the following property: the partial derivatives of $\nu(z)$ of order up to $N \ge 0$ exist and are absolutely integrable in $\R^d \setminus B(0, \varepsilon)$ for every $\varepsilon > 0$. Suppose that a bounded function $f$ is a regular harmonic function in an open set $D$. Then $f$ has continuous partial derivatives of order up to $N$. Furthermore, for every $\delta > 0$,
\begin{equation*}
 |D^j f(x)| \le c(\sigma, \nu, N, \delta) \|f\|_\infty
\end{equation*}
if $\dist(x, \R^d \setminus D) \ge \delta$ and $j$ is a multi-index with $|j| \le N$.
\end{thm}

\begin{remark}
\label{rem:reg}
\begin{enumerate}[label=(\alph*)]
\item The above result only accounts for smoothness of $f$ related to smoothness of~$\nu$. There may be additional smoothing effect related to `high activity' of the process: a non-vanishing diffusion component or sufficiently many small jumps.
\item\label{rem:reg:sbm} Recall that an isotropic unimodal L\'evy process is a subordinate Brownian motion if it is identical in law to the process $Y_{Z_t}$, where $Y_t$ is the standard Brownian motion in $\R^d$, $Z_t$ is a subordinator (a non-decreasing L\'evy process), and $Y_t$ and $Z_t$ are independent processes. Equivalently, $\nu(z)$ is the mixture of Gaussians, that is, $\nu(z) = \int_{(0, \infty)} e^{-s |z|^2} \mu(d s)$ for some measure $\mu$, such that $\int_{(0, \infty)} \min(1,s) \mu(ds) < \infty$. \par If $X_t$ is a subordinate Brownian motion, then, by Theorem~\ref{thm:reg}, harmonic functions are in fact smooth. Indeed, $\nu$ is a smooth function in $\R^d \setminus \{0\}$, and by differentiation under the integral sign one shows that partial derivatives of $\nu$ of order $j$ are bounded by $c(d, j) |z|^{-j} \nu(z/2)$. Smoothness of harmonic functions is a significant improvement of the result of~\cite{MR3161529}, especially for subordinate Brownian motions with `low activity', such as the variance gamma L\'evy process (see Example~\ref{ex:vg}) and geometric stable L\'evy processes (see Example~\ref{ex:gs}).
\item Let $\gener$ be the generator of $X_t$ (see~\eqref{eq:generator}). Consider the Cauchy problem
\begin{equation*}
 \begin{cases} \gener f = g_0 & \text{in $D$,} \\ \phantom{\gener} f = f_0 & \text{in $\R^d \setminus D$;} \end{cases}
\end{equation*}
here the equation $\gener f = g_0$ is understood in the weak sense. Then any solution $f$ of the above problem is at least as smooth as the worse of the following two: a harmonic function (the smoothness of which is described by Theorem~\ref{thm:reg}) and the convolution of $g_0$ with the Green function for $-\gener$ in full space $\R^d$ (which is just the potential kernel $U$ whenever it is not everywhere infinite; otherwise, a compensated potential kernel can be used). Indeed, in this case $\gener (f + g_0 * U) = \gener f - g_0 = 0$ in $D$, and thus $f + g_0 * U$ is harmonic in $D$.
\end{enumerate}
\end{remark}

Theorem~\ref{thm:green} provides a good estimate of the Green function of a ball away from the boundary. Estimates near the boundary typically rely on the scale-invariant boundary Harnack inequality. Although the boundary Harnack inequality is likely to hold for a wider class of isotropic unimodal L\'evy processes in sufficiently regular (for example, Lipschitz) domains, we choose to simply apply the result of~\cite{BKK}, which is valid for arbitrary domains, but requires some further regularity of the jumps.

\begin{thm}
\label{thm:bhi}
Let $R_\infty \in (0, \infty]$, and let $X_t$ be an isotropic unimodal L\'evy process which is not a compound Poisson process, and which contains no Gaussian component. Suppose that
\begin{equation}\label{eq:scaling}
 \frac{\nu(r_1)}{\nu(r_2)} \le M \left(\frac{r_1}{r_2}\right)^{\!\!-d - \alpha}
\end{equation}
holds with $\alpha < 2$ when $0 < r_1 < r_2 < 2 R_\infty$. If $R_\infty < \infty$, then suppose in addition that for every $r_0, \delta > 0$, the function
\begin{equation}\label{eq:harnack}
 \frac{\nu(r)}{\nu(r + \delta)}
\end{equation}
is bounded on $[r_0, \infty)$. Let $0 < r < R < R_\infty$ and let $p = (2 r + R)/3$, $q = (r + 2 R)/3$. Suppose that a non-negative function $f$ is a regular harmonic function in $D \cap B(x_0, R)$ (with respect to $X_t$), and it is equal to zero in $B(x_0, R) \setminus D$. Then there is a constant~$C$, which depends only on the characteristics of the process $X_t$, such that
\begin{equation*}
 f(x) \approx C \, \ex^x \tau_{D \cap B(x_0, p)} \int_{\R^d \setminus B(x_0, q)} f(y) \nu(y - x_0) dy
\end{equation*}
for $x \in D \cap B(x_0, r)$ (by $f \approx C g$ we understand $C^{-1} g \le f \le C g$). In particular, if $f$ and $g$ satisfy the above conditions, we have
\begin{equation}\label{eq:bhi}
 \sup_{x \in D \cap B(x_0, r)} \frac{f(x)}{g(x)} \le C^4 \inf_{x \in D \cap B(x_0, r)} \frac{f(x)}{g(x)} \, .
\end{equation}
\end{thm}

\begin{remark}
\begin{enumerate}[label=(\alph*)]
\item In the above result, the constant $C$ is \emph{scale-invariant}: it does not depend on $r$ and $R$, provided that the ratio $r / R$ is fixed.
\item A more general version, with an explicit constant $C$, is given in Theorem~\ref{thm:bhi:full}. This result covers also the case when the constant $C$ is not scale-invariant (for example, when $X_t$ contains a non-zero Gaussian component).
\item If $R_\infty = \infty$, then~\eqref{eq:scaling} implies boundedness of the function in~\eqref{eq:harnack}. This is why the latter condition is explicitly assumed to hold only when $R_\infty < \infty$.
\item By characteristics of the process $X_t$ we mean $R_\infty$, the parameters $\alpha$ and $M$ in~\eqref{eq:scaling}, and bounds on the function~\eqref{eq:harnack}. See the explicit expression for the constant $C$ in Theorem~\ref{thm:bhi:full} for further details.
\item Since constants are harmonic, the boundary Harnack inequality implies the usual Harnack inequality. Even this result is new for the class of processes considered in Theorem~\ref{thm:bhi}.
\item When the domain $D$ is sufficiently regular (for example, Lipschitz), the scale-invariant boundary Harnack inequality (in the form given in~\eqref{eq:bhi}) is known to hold for the Brownian motion, as well as some L\'evy processes with non-vanishing Gaussian component and jumps (see~\cite{MR2912450,MR3005005}). This suggests that, for regular domains $D$, Theorem~\ref{thm:bhi} could be extended to more general L\'evy processes. However, the boundary Harnack inequality~\eqref{eq:bhi} (and, in fact, even the usual Harnack inequality) does \emph{not} extend to arbitrary isotropic unimodal processes, as it is shown by Example~\ref{ex:hi}.
\end{enumerate}
\end{remark}

If we add a mild regularity condition on $\nu$, we obtain the existence of boundary limits of ratios of harmonic functions. This is an application of the main result of~\cite{JK15} and~\cite{KSV-MK}; we use the version stated in~\cite{JK15}.

\begin{thm}
\label{thm:martin}
Suppose that the assumptions of Theorem~\ref{thm:bhi} are satisfied, and in addition for all $r_0 > 0$ we have
\begin{equation}\label{eq:nu:continuity}
 \lim_{\delta \to 0^+} \frac{\nu(r)}{\nu(r + \delta)} = 1
\end{equation}
uniformly on every interval $[r_0, \infty)$. Suppose that $R$ and $f$ are as in Theorem~\ref{thm:bhi}. Then for each $z \in \partial D \cap B(x_0, R)$, a finite, non-negative limit
\begin{equation*}
 \lim_{\substack{x \to z \\ x \in D}} \frac{f(x)}{\ex^x \tau_{D \cap B(x_0, R)}}
\end{equation*}
exists. In particular, if $f$ and $g$ satisfy the conditions of Theorem~\ref{thm:bhi}, then $f / g$ has boundary limits on $\partial D \cap B(x_0, R)$.
\end{thm}

\begin{remark}
\begin{enumerate}[label=(\alph*)]
\item As a consequence of Theorem~\ref{thm:martin}, non-negative functions harmonic with respect to $X_t$ in $D$ admit Martin representation, with Martin boundary identified with a subset of the Euclidean boundary $\partial D$. We refer to~\cite{JK15} for more details.
\item Theorems~\ref{thm:bhi} and~\ref{thm:martin} extend to isotropic unimodal L\'evy processes killed by a continuous multiplicative functional (which corresponds to adding a Schr\"odinger potential to the generator), see Theorem~5.4 in~\cite{BKK}.
\end{enumerate}
\end{remark}

As an application of the boundary Harnack inequality and bounds for the potential kernel and Green function of a ball, we find a sharp, two-sided estimate of the Green function of a half-space. A similar method can apparently be used to find similar bounds for more general sets with $C^{1,1}$ boundary (in particular, for balls and their complements).

In the following result, we assume that the radial profile of $\nu(x)$ is $O$-regularly varying at zero and at infinity with lower indices greater than $-d - 2$ (see Remark~\ref{rem:orv}). Relaxing this condition would require an appropriate extension of the boundary Harnack inequality. On the other hand, virtually all known explicit bounds for Green functions of unbounded sets require not only that lower indices are greater than $-d - 2$, but also that upper indices are less than $-d$ (that is, the reversed inequality~\eqref{eq:scaling} holds for some $M > 0$ and $\alpha > 0$); we refer here to Theorem~5.11 in~\cite{MR3131293}, Theorem~7.1 in~\cite{CK2016} and Theorem~5.8 in~\cite{BGR2}. The only exception known to the authors is a result for one-dimensional geometric stable processes, see Theorem~4.4 in~\cite{MR3007664} and Example~\ref{ex:gs}.

\begin{thm}
\label{thm:halfspace}
Suppose that $d \ge 3$, and let $X_t$ be an isotropic unimodal L\'evy process in $\R^d$, which is not a compound Poisson process and which contains no Gaussian component. Suppose that~\eqref{eq:scaling} holds with $\alpha < 2$ when $0 < r_1 < r_2$ (that is, with $R_\infty = \infty$). Let $H$ be a half-space. Then
\begin{equation}\label{eq:halfspace}
 G_H(x, y) \approx c(d, \alpha, M) \, \frac{h(\delta_x)}{h(\delta_x + r)} \, \frac{h(\delta_y)}{h(\delta_y + r)} \, \frac{K(r)}{r^d (K(r) + L(r))^2}
\end{equation}
for all $x, y \in H$, where $r = |x - y|$, $\delta_x = \dist(x, \R^d \setminus H)$, $\delta_y = \dist(y, \R^d \setminus H)$ and $h(s) = 1 / \sqrt{K(s) + L(s)}$ (by $f \approx C g$ we understand $C^{-1} g \le f \le C g$).
\end{thm}

\begin{remark}
Once it is proved that $h(\delta_x)$ describes the boundary decay rate of regular harmonic functions in an open set satisfying the interior and exterior ball conditions (that is, a set with $C^{1,1}$ boundary), Theorem~\ref{thm:halfspace} can be extended to this class of sets.
\end{remark}

\begin{remark}\label{rem:orv}
Our results rely on the comparability condition~\eqref{eq:scaling} on the density of the L\'evy measure, known as \emph{$O$-regular variation} or \emph{lower scaling condition}. It is worthwhile to explain its role in our results.
\begin{enumerate}[label=(\alph*)]
\item If $A (r_2 / r_1)^a \le \varphi(r_2) / \varphi(r_1) \le B (r_2 / r_1)^b$ when $R_0 < r_1 < r_2$ (for some constants $A, B > 0$, $a, b \in \R$), then $\varphi$ is said to be \emph{$O$-regularly varying} at infinity. The supremum of the numbers $a$ for which the lower bound is satisfied for some $m > 0$ is called the \emph{lower index} of $\varphi$ at infinity, while the analogous infimum of the numbers $b$ --- the \emph{upper index} of $\varphi$. These indices are defined for an arbitrary positive function $\varphi$, although in general they can be infinite. Clearly, since $\nu$ is non-increasing, its upper index at infinity is at most $0$. Condition~\eqref{eq:scaling} (for large $r_1$, $r_2$) is thus equivalent to saying that $\nu$ is $O$-regularly varying at infinity, with lower index at infinity greater than $-d - 2$.

The notion of $O$-regular variation at zero is defined in a similar way. Thus, assuming that~\eqref{eq:scaling} holds when $0 < r_1 < r_2$ (that is, $R_\infty = \infty$) is equivalent to assuming that $\nu$ is $O$-regularly varying at zero and at infinity, with lower indices greater than $-d - 2$.
\item\label{rem:orv:constant} If $\sigma^2 = 0$ (that is, $X_t$ has no Gaussian component) and~\eqref{eq:scaling} holds when $0 < r_1 < r_2 < R_\infty$ with $\alpha < 2$, then $K(r) \le c(d, \alpha, M) r^d \nu(r)$ when $0 < r < R_\infty$. In particular, Theorems~\ref{thm:sup}, \ref{thm:ret} and~\ref{thm:pot} (and partially also Theorem~\ref{thm:green}) provide sharp two-sided estimates in this case.
\item\label{rem:orv:karamata} By a Tauberian-type theorem of Karamata, the converse is also true: if $K(r)$ is comparable with $r^d \nu(r)$ when $r > 0$, then $\nu$ is $O$-regularly varying at zero and at infinity, with lower indices greater than $-d - 2$. This follows rather easily from the results of~\cite{MR0466438}, as it is discussed in Appendix~\ref{sec:app:orv} (see Corollary~\ref{cor:orv:unimodal}).
\item\label{rem:orv:fourier} Finally, many former works give conditions on the process $X_t$ in terms of its characteristic exponent $\Psi$. Although the characteristic exponent is not used in this article, it is worth noting that for $\alpha \in [0, 2)$, $O$-regular variation of $\nu$ at infinity with lower index $-d - \alpha$ implies $O$-regular variation of (the profile function of) $\Psi$ at zero with upper index $\alpha$. The converse is false, unless an extra Taberian-type condition is imposed. In fact, if the upper index of $\Psi$ at zero is $\alpha \in [0, 2)$, then the lower index of $L(r)$ at infinity is equal to $-\alpha$, but this does not imply any bound on the lower index of $\nu$. This subject is discussed in Appendix~\ref{sec:app:orv}, see Propositions~\ref{prop:lk} and~\ref{prop:lknu}.
\end{enumerate}
\end{remark}

We conclude the introduction with a brief description of the structure of this article. In Section~\ref{sec:pre} we introduce the notation and discuss some known results. The proofs of Theorems~\ref{thm:ret}, \ref{thm:pot}, \ref{thm:green} and~\ref{thm:sup} are given in Section~\ref{sec:general}. The most difficult step here is Lemma~\ref{lem:ret}, which provides the upper bound of Theorem~\ref{thm:ret} for the probability of hitting a ball when the distance between the ball and the starting point $X_0$ is comparable to the radius of the ball. Boundary Harnack inequality and related results are discussed in Section~\ref{sec:bhi}, where we prove Theorems~\ref{thm:bhi}, \ref{thm:martin} and~\ref{thm:halfspace}. A more general version of the boundary Harnack inequality for L\'evy processes is also stated there, see Theorem~\ref{thm:bhi:full}. Section~\ref{sec:ex} contains a number of examples. Finally, in Appendix~\ref{sec:app:orv} we discuss the relation of comparability condition~\ref{eq:scaling} with the notion of $O$-regularly varying functions, and in Appendix~\ref{sec:app:sup} we provide a proof of a result on averaging certain unimodal kernels; both of these ideas are likely well known, but we could not find an appropriate reference in existing literature.

%
%

\section{Preliminary results}
\label{sec:pre}

In this section we discuss the notation and assumptions, and we prove general estimates for isotropic unimodal L\'evy processes. Throughout this article, $d$ denotes the dimension of the Euclidean space $\R^d$. We write $B(x, r)$ for an open ball in $\R^d$, and we let $B_r = B(0, r)$. By $c(d, \ldots)$ we denote a generic (typically: large) positive constant, that depends only on the listed parameters $d, \ldots$ All subsets of $\R^d$ and functions defined on $\R^d$ in this article are assumed to be Borel.

\subsection{Isotropic unimodal L\'evy processes}

We consider a L\'evy process $X_t$ in $\R^d$; that is, $X_t$ is a non-constant stochastic process with independent and identically increments and càdlàg paths. By $\pr^x$ we denote the probability corresponding to the process starting at $x \in \R^d$, that is, $\pr^x(X_0 = x) = 1$; $\ex^x$ is the corresponding expectation. The process $X_t$ is said to be \emph{isotropic} if the distribution $P_t(dz)$ of $X_t - X_0$ is invariant under rotations; $X_t$ is \emph{isotropic unimodal} if in addition $P_t(dz)$ is a unimodal measure: it may have an atom at the origin, and it is absolutely continuous in $\R^d \setminus \{0\}$ with respect to the Lebesgue measure, with an (isotropic) density function that has a non-increasing radial profile. Equivalently, $X_t$ is isotropic unimodal if its L\'evy measure is an isotropic unimodal measure, it has no drift, and the Gaussian component has a covariance matrix $\sigma^2 \Id$ (here $\Id$ is the identity matrix and $\sigma^2$ will be called the Gaussian coefficient). Clearly, the potential measure $U(E) = \int_0^\infty P_t(E) dt$ is also an isotropic unimodal (possibly everywhere infinite) measure. For more information on isotropic unimodal L\'evy processes, see~\cite{BGR,TG,MR705619}.

To simplify the notation, throughout the article we commonly use the same symbol to denote a measure, its density function, and its radial profile. In particular, we identify $\nu(z)$ and $\nu(|z|)$, and we write $P_t(z)$ or $P_t(|z|)$ for the density function of $P_t(dz)$ on $\R^d \setminus \{0\}$. We also sometimes use $X(t)$ instead of $X_t$ to avoid nested subscripts.

For an open set $D$, let $\tau_D$ be the time of first exit from $D$, $\tau_D = \inf \{ t > 0 : X_t \notin D \}$. By $P^D_t(x, dy)$ and $G_D(x, dy)$ we denote the transition kernels and the Green function for the process $X_t$ killed at $\tau_D$: $P_t^D(x, E) = \pr^x(X_t \in E; t < \tau_D)$, and $G_D(x, E) = \int_0^\infty P_t^D(x, E) dt$. Clearly, if $D = \R^d$, then $P_t^D(0, dy) = P_t(dy)$ and $G_D(0, dy) = U(dy)$. It is also worth noting that $G_D(x, dy)$ is the occupation measure:
\begin{equation*}
 G_D(x, E) = \ex^x \int_0^{\tau_D} \ind_E(X_t) dt .
\end{equation*}
In particular, $G_D(x, D) = \ex^x \tau_D$.

A function $f$ is said to be a \emph{regular harmonic} with respect to $X_t$ in an open set $D$ if $f$ has the mean-value property $f(x) = \ex^x f(X(\tau_D))$ for all $x \in D$. If $f$ has the mean-value property in every bounded open set whose closure is contained in $D$, then $f$ is said to be a \emph{harmonic function} in $D$. For further discussion, see Section~VII.3 in~\cite{MR850715} and Section~24 in~\cite{MR0346919}.

The \emph{Feller generator} of $X_t$ is an operator $\gener$, whose domain is the set of those $f \in C_0(\R^d)$ (continuous functions convergent to zero at infinity) for which the limit in the definition of $\gener f$,
\begin{equation*}
 \gener f(x) = \lim_{t \to 0^+} \frac{1}{t} \left( \int_{\R^d} f(x + z) P_t(dz) - f(x) \right) ,
\end{equation*}
exists uniformly on $\R^d$. If $X_t$ is an isotropic unimodal L\'evy process, $f$ is twice differentiable, and $f$ and all partial derivatives of $f$ of order $1$ and $2$ are in $C_0$, then $f$ is in the domain of $\gener$, and
\begin{equation}\label{eq:generator}
\begin{aligned}
 \gener f(x) & = \sigma^2 \Delta f(x) + \lim_{\varepsilon \to 0^+} \int_{\R^d \setminus B_\varepsilon} (f(x + z) - f(x)) \nu(z) dz \\
 & = \sigma^2 \Delta f(x) + \int_{\R^d } (f(x + z) - f(x) - z \cdot \nabla f(x) \ind_B(z)) \nu(z) dz .
\end{aligned}
\end{equation}
For further information, see Section~31 in~\cite{MR1739520}.

One of the fundamental tools in the theory of L\'evy (or, more generally, Feller) processes is \emph{Dynkin's formula}, which states that for an arbitrary Markov time $\tau$ such that $\ex^x \tau < \infty$ and for any function $f$ in the domain of the Feller generator $\gener$, we have
\begin{equation}\label{eq:dynkin}
 \ex^x f(X_\tau) - f(x) = \ex^x \int_0^\tau \gener f(X_t) dt .
\end{equation}
Typically, one considers $\tau$ to be equal to the time of first exit from an open set $D$. In this case, by considering $f$ equal to zero in $D$, one easily gets (a simplified version of) the \emph{Ikeda--Watanabe formula}, or the L\'evy system formula for L\'evy processes:
\begin{equation*}
 \ex^x f(X_{\tau_D}) = \int_D \left( \int_{\R^d} f(z) \nu(z - y) dz \right) G_D(x, dy)
\end{equation*}
whenever $f$ is equal to zero in $\overline{D}$ and the integrals on either side converge. In particular, if $P_D(x, dz)$ denotes the distribution of $X(\tau_D)$ with respect to $\pr^x$, then the measure $P_D(x, dz)$ may contain a singular part on $\partial D$, and it has a density function $P_D(x, z)$ on $\R^d \setminus \overline{D}$ (sometimes called the Poisson kernel, by analogy with the classical case), given by
\begin{equation}\label{eq:iw}
 P_D(x, z) = \int_D \nu(z - y) G_D(x, dy) ,
\end{equation}
where $x \in D$ and $z \in \R^d \setminus \overline{D}$.

To make the arguments easier to read, we tacitly assume that $X_t$ is not a compound Poisson process, and thus $P_t(dz)$, $U(dz)$, $P_t^D(x, dy)$ and $G_D(x, dy)$ have no atom at the origin. For example, we will write $P_D(x, z) = \int_D G_D(x, y) \nu(z - y) dy$ instead of formally more precise~\eqref{eq:iw}. Extension to the general case should present no difficulty to the reader: it essentially consists of replacing $U(z) dz$ by $U(dz)$ etc.

In theorems and propositions, we carefully formulate all assumptions. In lemmas and proofs, however, we always assume that $X_t$ is an isotropic unimodal L\'evy process in $\R^d$, and we use freely the notation introduced above. In particular, $\sigma^2$ is the Gaussian coefficient of $X_t$, $\nu$ is the density function of the L\'evy measure of $X_t$, and $U$ is the potential kernel of $X_t$. We also use the functions $K(r)$ and $L(r)$, given by~\eqref{eq:KL}; recall that
\begin{equation*}
\begin{aligned}
 K(r) & = \frac{\sigma^2 d}{r^2} + \int_{B_r} \frac{|z|^2}{r^2} \nu(z) dz = \frac{\sigma^2 d}{r^2} + r^d \int_{B_1} |y|^2 \nu(r y) dy , \\
 L(r) & = \int_{\R^d \setminus B_r} \nu(z) dz = r^d \int_{\R^d \setminus B_1} \nu(r y) dy .
\end{aligned}
\end{equation*}
These quantities describe the `activity' of $X_t$ and intensity of large jumps, respectively, and they were introduced by Pruitt in~\cite{MR632968} in order to study the mean exit time from a ball: we have
\begin{equation}\label{eq:pruitt}
 \frac{1}{c(d)} \, \frac{1}{K(r) + L(r)} \le \ex^0 \tau_{B_r} \le c(d) \, \frac{1}{K(r) + L(r)} \, ,
\end{equation}
see p.~954, Theorem~1 and~(3.2) in~\cite{MR632968}. Whenever possible, we state our estimates in terms of $K(r)$ and $L(r)$.

We note that the estimates that depend only on $K(r) + L(r)$ (like~\eqref{eq:pruitt}) can be rewritten in terms of the characteristic (L\'evy--Khintchine) exponent $\Psi(z)$, since if $z \in \R^d$ and $r = 1 / |z|$, then, by Corollary~3 in~\cite{BGR},
\begin{equation}\label{eq:lk}
 \frac{1}{c(d)} (K(r) + L(r)) \le \Psi(z) \le c(d) (K(r) + L(r)) .
\end{equation}
Here $\Psi$ is related to the characteristic function of $X_t$,
\begin{equation*}
 e^{-t \Psi(z)} = \ex^0 e^{i z \cdot X_t} = \int_{\R^d} e^{i z \cdot y} P_t(dy) ,
\end{equation*}
and by the L\'evy--Khintchine formula, for an isotropic unimodal L\'evy process $X_t$,
\begin{equation*}
 \Psi(z) = \sigma^2 |z|^2 + \int_{\R^d} (1 - \cos(z \cdot y)) \nu(y) dy .
\end{equation*}
The function $\Psi$, however important in the theory of L\'evy processes, will not be used below.

\subsection{Basic properties of $K(r)$ and $L(r)$}

We collect some very simple properties of $K(r)$ and $L(r)$. By definition, 
\begin{equation*}
 K(r) + L(r) = \frac{\sigma^2 d}{r^2} + \int_{\R^d} \min\left(\frac{|z|^2}{r^2}, 1\right) \nu(z) dz = \frac{\sigma^2 d}{r^2} + r^d \int_{\R^d} \min\left(|y|^2, 1\right) \nu(r y) dy .
\end{equation*}
For any $a \ge 1$ we have $a^{-2} \min(|z|^2 / r^2, 1) \le \min(|z|^2 / (a^2 r^2), 1) \le \min(|z|^2  r^2, 1)$. Thus,
\begin{equation*}
 \frac{K(r) + L(r)}{a^2} \le K(a r) + L(a r) \le K(r) + L(r) .
\end{equation*}
It is also easy to see that for any $a \ge 1$,
\begin{equation*}
\begin{aligned}
 \frac{K(r)}{a^2} & \le K(a r) \le a^d K(r) , & \qquad L(a r) & \le L(r) ,
\end{aligned}
\end{equation*}
and
\begin{equation*}
 K(r) \ge \frac{1}{c(d)} \, r^d \nu(r) .
\end{equation*}
Finally, as it was noted in Remark~\ref{rem:orv}, if $\sigma^2 = 0$ and~\eqref{eq:scaling} holds with $\alpha < 2$ when $0 < r_1 < r_2 < R$, then
\begin{equation*}
 \frac{1}{c(d)} \, r^d \nu(r) \le K(r) \le c(d, \alpha, M) r^d \nu(r) .
\end{equation*}
These properties of $K(r)$ and $L(r)$ are used below without further comment. Note that $K(r)$ and $K(r) + L(r)$ are $O$-regularly varying functions both at zero and at infinity.

\subsection{Generator}

The following lemma is a straightforward modification of the proof of Theorem~31.5 in~\cite{MR1739520}.

\begin{lem}\label{lem:generator}
Let $f$ be a non-negative, smooth function supported in a ball of radius $R > 0$. If $0 < r \le R$, then 
\begin{equation*}
 \gener f(x) \le c(d) \|f''\|_\infty r^2 K(r) + \|f\|_\infty (L(r) - L(R + r)) ,
\end{equation*}
where $f''$ denotes the matrix of second order partial derivatives of $f$, the norm $\|f''\|_\infty$ is the maximum of the supremum norms of the entries of the matrix $f''$, and $K(r)$ and $L(r)$ are given by~\eqref{eq:KL}. In particular,
\begin{equation*}
 \gener f(x) \le c(d) \|f''\|_\infty R^2 K(R) .
\end{equation*}
\end{lem}

\begin{proof}
The first statement follows from the representation~\eqref{eq:generator} of the generator of $X_t$, symmetry of $X_t$ and second-order Taylor approximation for $f$. By translation invariance, with no loss of generality we may assume that $f$ is equal to zero outside $B_R$. We have
\begin{equation*}
\begin{aligned}
 \gener f(x) & = \sigma^2 \Delta f(x) + \int_{B_r} (f(x + z) - f(x) - \nabla f(x) \cdot z) \nu(z) dz \\ & \hspace*{13em} + \int_{\R^d \setminus B_r} (f(x + z) - f(x)) \nu(z) dz \\
 & \le \sigma^2 \|f''\|_\infty + c(d) \|f''\|_\infty \int_{B_r} |z|^2 \nu(z) dz + 2 \|f\|_\infty \, \nu(B(-x, R) \setminus B_r) .
\end{aligned}
\end{equation*}
Recall that the measure $\nu$ is isotropic and unimodal. By a rearrangement argument, $\nu(B(-x, R) \setminus B_r) \le \nu(B_{R + r} \setminus B_r)$; indeed, $|B(-x, R) \setminus B_r| \le |B_{R + r} \setminus B_r|$, and the latter set maximizes the measure $\nu$ among the sets of equal Lebesgue measure that are disjoint from $B_r$.

To prove the second statement, we use the first one with $r = R$: we have $\|f\|_\infty \le c(d) \|f''\|_\infty R^2$ and $K(R) + (L(R) - L(R + R)) \le 4 K(2 R) \le 2^{d + 2} K(R)$.
\end{proof}

\subsection{Known estimates}

If $d \ge 3$, then it is known that
\begin{equation}\label{eq:pot:bound}
 \frac{1}{c(d)} \, \frac{1}{K(r) + L(r)} \le \int_{B_r} U(z) dz \le c(d) \, \frac{1}{K(r) + L(r)} \, ;
\end{equation}
the lower bound holds in arbitrary dimension $d \ge 1$ (as mentioned above, here and below we write $U(z) dz$ instead of more formal $U(dz)$). We remark that the lower bound follows from the estimate $U(z) \ge G_{B_r}(0, z)$ and Pruitt's estimate~\eqref{eq:pruitt}; the upper bound is proved by comparing $\ind_{B_r}$ with a Gauss--Weierstrass kernel and using Fourier transform, see Proposition~2 in~\cite{TG}.

By unimodality,~\eqref{eq:pot:bound} implies that
\begin{equation}\label{eq:pot:bound:u}
 U(r) \le \frac{1}{|B_r|} \int_{B_r} U(z) dz \le c(d) \, \frac{1}{r^d (K(r) + L(r))} \, .
\end{equation}
A somewhat similar upper bound is also available for $P_t(r)$: we have
\begin{equation}\label{eq:pt:bound:u}
 P_t(r) \le c(d) \, \frac{t K(r)}{r^d} \, .
\end{equation}
We note that this estimate is a simple consequence of Lemma~\ref{lem:generator}, Dynkin's formula~\eqref{eq:dynkin} and unimodality of $P_t$; see Theorem~5.4 in~\cite{GRT}. Finally, the Ikeda--Watanabe formula~\eqref{eq:iw} implies the following simple estimate of the Poisson kernel for a ball.

\begin{lem}\label{lem:pb:bound}
If $0 < r < |z|$, then
\begin{equation}\label{eq:pb:bound}
 \frac{1}{c(d)} \, \frac{\nu(z)}{K(r) + L(r)} \le P_{B_r}(0, z) \le c(d) \, \frac{\nu(|z| - r)}{K(r) + L(r)} \, .
\end{equation}
\end{lem}

\begin{proof}
Let $z_0$ be a point on the line segment $[0, z]$ such that $|z| = r$. By the Ikeda--Watanabe formula~\eqref{eq:iw},
\begin{equation*}
\begin{aligned}
 P_{B_r}(0, z) & = \int_{B_r} G_{B_r}(0, y) \nu(z - y) dy \\
 & \ge \nu(z) \int_{B_r \cap \overline{B}(z, |z|)} G_{B_r}(0, y) dy \ge \nu(z) \int_{B_r \cap \overline{B}(z_0, r)} G_{B_r}(0, y) dy .
\end{aligned}
\end{equation*}
Since $G_{B_r}(0, y)$ is an isotropic kernel, and a constant number $c(d)$ of rotations (i.e.\@ images under orthogonal transformations of $\R^d$) of $B_r \cap \overline{B}(z_0, r)$ can cover $B_r$, we have
\begin{equation*}
 P_{B_r}(0, z) \ge \frac{1}{c(d)} \, \nu(z) \int_{B_r} G_{B_r}(0, y) dy = \frac{1}{c(d)} \, \nu(z) \ex^0 \tau_{B_r} .
\end{equation*}
The desired lower bound follows from Pruitt's estimate~\eqref{eq:pruitt}. The upper bound is even simpler: we have
\begin{equation*}
 P_{B_r}(0, z) = \int_{B_r} G_{B_r}(0, y) \nu(z - y) dy \le \nu(|z| - r) \int_{B_r} G_{B_r}(0, y) dy = \nu(|z| - r) \ex^0 \tau_{B_r} ,
\end{equation*}
and it remains to use Pruitt's estimate~\eqref{eq:pruitt}.
\end{proof}

\subsection{Unimodality}

As remarked above, for an isotropic unimodal L\'evy process $X_t$, the L\'evy measure, the potential kernel and transition probabilities are unimodal measures. It is not very difficult to prove similar properties for the objects related to the process killed at the time of first exit from a ball.

\begin{prop}\label{prop:pb:unimodal}
Let $X_t$ be an isotropic unimodal L\'evy process in $\R^d$, and let $B$ be the unit ball in $\R^d$. Then the transition probability $P^B_t(0, dy)$ of $X_t$ started at $0$ and killed at the time of first exit from $B$ is an isotropic unimodal measure.
\end{prop}

\begin{proof}
Let the sequence of measures $P_t^{(n)}$ be defined recursively:
\begin{align*}
 P_t^{(1)} & = \ind_{\overline{B}} \, P_t , & P_t^{(n+1)} & = \ind_{\overline{B}} \, (P_t^{(n)} * P_t) .
\end{align*}
By induction, $P_t^{(n)}$ is an isotropic unimodal measure; indeed, convolution and multiplication by $\ind_{\overline{B}}$ preserves this property. Furthermore
\begin{align*}
 P_{t/n}^{(n)}(E) & = \pr^0(X_t \in E; \, X_{t/n}, X_{2t/n}, \ldots, X_{nt/n} \in \overline{B}) .
\end{align*}
By right-continuity of paths, $P_{t/2^n}^{(2^n)}(E)$ is a non-increasing sequence convergent to $P^{\overline{B}}_t(0, E)$. The class of isotropic unimodal measures is closed in the vague topology, and thus $P^{\overline{B}}_t(0, dy)$ is an isotropic unimodal measure.

The same result is clearly true when $B$ is replaced with $B_r$. Furthermore, $P^B_t(0, dy)$ is the vague limit of $P^{\overline{B}_r}_t(0, dy)$ as $r \to 1^-$, as a result of quasi-left continuity of $X_t$ (see Theorem~40.12 in~\cite{MR1739520}). This proves the desired result.
\end{proof}

\begin{cor}\label{cor:pb:unimodal}
Let $X_t$ be an isotropic unimodal L\'evy process in $\R^d$, and let $B$ be the unit ball in $\R^d$. Then the Green kernel of $X_t$ in $B$ with pole at $0$, $G_B(0, dy)$, is an isotropic unimodal measure. Furthermore, the distribution $P_B(0, dz)$ of $X_t$ started at $0$ and stopped at the time of first exit from $B$ is \emph{unimodal in $\R^d \setminus B$} in the following sense: $P_B(0, dz)$ is an isotropic measure, which may contain a singular part on $\partial B$, and which is absolutely continuous with respect to the Lebesgue measure on $\R^d \setminus \overline{B}$, with a radial density function that has a non-increasing profile function on $(1, \infty)$.
\end{cor}

\begin{proof}
For the first part, it suffices to recall that $G_B(0, E) = \int_0^\infty P^B_t(0, E) dt$. The second one follows from the Ikeda--Watanabe formula~\eqref{eq:iw}: in $\R^d \setminus \overline{B}$, $P_B(0, dz)$ is equal to the convolution of $G_B(0, dy)$ and $\nu(z)$. (This is an abuse of the classical theorem, stating the the convolution of two isotropic unimodal measures is isotropic unimodal: here the convolution can be an everywhere infinite measure in $B$. Nevertheless, it is isotropic and unimodal by exactly the same argument as in the finite case.)
\end{proof}

The above results extend trivially to balls $B_r$ of arbitrary radius. Furthermore, the proofs only depend on the fact that $X_0$ has an isotropic unimodal distribution: this observation will be used later in the proof of Theorem~\ref{thm:sup}.

%
%

\section{General estimates}
\label{sec:general}

In this section we prove main results that do not require any further assumptions.

\subsection{Probability of hitting a nearby ball}

The following preliminary upper estimate of the probability that the process $X_t$ ever hits a given ball is a crucial step in the proof of the upper bound of the potential kernel. It also implies the upper bound of Theorem~\ref{thm:ret} when $|x|$ is comparable with $r$, that is, when the starting point $x$ is nearby the ball $B_r$. The general case involves the estimates of the potential kernel, and for this reason it is considered after the proof of Theorem~\ref{thm:pot}.

\begin{lem}\label{lem:ret}
Suppose that $d \ge 3$, and let $X_t$ be an isotropic unimodal L\'evy process in~$\R^d$.  Then
\begin{equation}\label{eq:ret:u}
 \pr^x(T_r < \infty) \le c(d) \, \frac{K(r)}{K(r) + L(r)}
\end{equation}
when $|x| \ge 2 r$.
\end{lem}

\begin{proof}
We begin with the following (slightly informal) observation: the function $f(x) = \min(1, (r / |x|)^{d - 2})$ is a superharmonic function in $\R^d$ in the classical sense (that is, with respect to the Brownian motion). It follows that the mean value of $f$ over any sphere $\partial B(x, s)$ is not greater than $f(x)$. Integration with respect to $s$ implies that
\begin{equation*}
 \gener f(x) = \sigma^2 \Delta f(x) + \lim_{\varepsilon \to 0^+} \int_\varepsilon^\infty \left(\int_{\partial B(x, s)} (f(z) - f(x)) \sigma(dz) \right) \nu(s) ds \le 0
\end{equation*}
for all $x \in \R^d$, which means that $f$ is also a superharmonic function with respect to $X_t$ (here $\sigma(dz)$ denotes the surface measure). Thus, $\pr^x(T_r < \infty) \le f(x)$. This bound is insufficient for our needs,but it presents the main idea of the proof: we need an appropriate superharmonic function. We will use a smooth approximation to the function $f(x) = \min(1, (R / |x|)^{d - 2}) + \ind_{B_r}(x)$.

\emph{Step 1.}
Let $a, b \in (0, 1)$ be constants depending only on $d$, the value of which is to be fixed at a later stage of the proof. Clearly, if $a L(r) \le 5^d K(r)$, then $K(r) / (K(r) + L(r)) \ge a / (a + 5^d)$, and so~\eqref{eq:ret:u} holds trivially with constant $c(d) = (a + 5^d) / a$. Therefore, in the remaining part of the proof we consider only the case $a L(r) > 5^d K(r)$.

\emph{Step 2.}
In order to construct the appropriate superharmonic function, we introduce auxiliary functions $g$ and $h$. The function $g$ approximates $\ind_{B_r}$: it takes values in $[0, 1]$, it is equal to $1$ in $B_r$, it is equal to zero in $\R^d \setminus B_{2 r}$, and it satisfies $\|g''\|_\infty \le c(d) / r^2$ (with the notation of Lemma~\ref{lem:generator}). We define
\begin{equation*}
 R = \left(\frac{a L(r)}{K(r)}\right)^{\!\! 1/d} r .
\end{equation*}
Note that as we only consider the case $a L(r) > 5^d K(r)$, we have $R > 5 r$. We also let
\begin{equation*}
 h_R(x) = \min\left(1, \frac{R^{d - 2}}{|x|^{d - 2}}\right) .
\end{equation*} 
For technical reasons, it is convenient to define $h$ to be a smooth approximation to $h_R$, although using $h = h_R$ would also be possible. Let $\kappa$ be a smooth non-negative radial function which is equal to zero outside $B_r$, and such that $\int_{B_r} \kappa(z) dz = 1$, and let
\begin{equation*}
 h(x) = h_R * \kappa(x) = \int_{B_r} h_R(x - y) \kappa(y) dy .
\end{equation*}

\emph{Step 3.}
We will examine the function
\begin{equation*}
 f(x) = b L(r) g(x) + K(r) h(x) .
\end{equation*}
More precisely, below we prove that it is possible to choose $a$ and $b$ in such a way that $f$ is a superharmonic function in $\R^d$ with respect to $X_t$. The theorem follows easily once this is done. Indeed, since $f$ is positive on $\R^d$ and constant on $\overline{B}_r$, whenever $r < |x| < s$, we have
\begin{equation*}
 f(x) \ge \ex^x f(X(\tau_{B_s \setminus \overline{B}_r})) \ge \ex^x (f(X_{T_r}) \ind_{\{T_r < \tau_{B_s}\}}) .
\end{equation*}
Passing to the limit as $s \to \infty$ gives
\begin{equation*}
 f(x) \ge \ex^x (f(X_{T_r}) \ind_{\{T_r < \infty\}}) = f(0) \pr^x(T_r < \infty) .
\end{equation*}
Finally, $f(0) = K(r) + b L(r)$, and if $|x| \ge 2 r$, then $f(x) = K(r) h(x) \le K(r)$. We conclude that
\begin{equation*}
 \pr^x(T_r < \infty) \le \frac{K(r)}{K(r) + b L(r)} \, ,
\end{equation*}
which implies desired bound (with constant $c(d) = 1 / b$). Therefore, it is sufficient to show that (for appropriate $a$ and $b$) $f$ is a superharmonic function with respect to $X_t$ in $\R^d \setminus \overline{B}_r$. We will achieve that goal by estimating $\gener f(x)$.

\emph{Step 4.}
First, we collect some simple properties of $h$. Since $\kappa$ and $h_R$ are radial, $h = h_R * \kappa$ is radial too. The function $h_R$ is superharmonic in $\R^d$ and harmonic in $\R^d \setminus \partial B_R$ in the classical sense. Hence, the mean value of $h_R$ over a sphere $\partial B(x, s)$ is not greater than $h_R(x)$, and it is equal to $h_R(x)$ if $B(x, s)$ does not intersect $\partial B_R$. Integrating this inequality or equality with respect to $s \in [0, r]$, we find that $h(x) \le h_R(x)$ for all $x \in \R^d$, and $h(x) = h_R(x)$ when $|x| \le R - r$ or $|x| \ge R + r$. Furthermore, $h(x)$ is greater than the infimum of $h_R$ over $B(x, r)$, that is, $h(x) \ge \min(1, R^{d - 2} / (|x| + r)^{d - 2})$ for all $x \in \R^d$. Finally, $h = h_R * \kappa$ is superharmonic in $\R^d$ in the classical sense, and thus the mean value of $h$ over a sphere $\partial B(x, s)$ is not greater than $h(x)$.

\emph{Step 5.}
We consider first the case $|x| \le R + r$. By Lemma~\ref{lem:generator}, for all $x \in \R^d$,
\begin{equation}\label{eq:ret:geng1}
 \gener g(x) \le c(d) K(r) .
\end{equation}
To estimate $\gener h(x)$, recall that $\Delta h(x) \le 0$, and that the mean value of $h$ over a sphere $\partial B(x, s)$ is not grater than $h(x)$. Thus, by the definition~\eqref{eq:generator} of $\gener h(x)$, we have
\begin{equation*}
 \gener h(x) \le \int_{\R^d \setminus B(x, 4 R)} (h(y) - h(x)) \nu(y - x) dy .
\end{equation*}
Observe that $h(x) \ge R^{d-2} / (R + 2r)^{d - 2} \ge (5 / 7)^{d - 2}$, and that if $|y - x| \ge 4 R$, then $|y| \ge |y - x| - |x| \ge 4 R - (R + r) > 2 R$, so that $h(y) = h_R(y) = (R / |y|)^{d - 2} \le (1/2)^{d - 2}$. It follows that
\begin{equation}\label{eq:ret:genh1}
 \gener h(x) \le \int_{\R^d \setminus B(x, 4 R)} ((1/2)^{d - 2} - (5/7)^{d - 2}) \nu(y - x) dy = -\frac{1}{c(d)} \, L(4 R).
\end{equation}
By~\eqref{eq:ret:geng1} and~\eqref{eq:ret:genh1}, we have
\begin{equation*}
 \gener f(x) = b L(r) \gener g(x) + K(r) \gener h(x) \le c(d) b K(r) L(r) - \frac{1}{c(d)} \, K(r) L(4 R) .
\end{equation*}
In order to prove that the right-hand side is negative, observe that
\begin{equation*}
 L(r) = L(4 R) + (L(r) - L(4 R)) \le L(4 R) + |B_{4 R}| \nu(r) = L(4 R) + c(d) \nu(r) R^d ,
\end{equation*}
and so, by the definition of $R$,
\begin{equation*}
 L(r) \le L(4 R) + c(d) a \, \frac{r^d \nu(r) L(r)}{K(r)} \le L(4 R) + c(d) a L(r) .
\end{equation*}
It follows that if $a$ satisfies $c(d) a < 1/2$, then $L(r) \le 2 L(4 R)$, and hence
\begin{equation*}
 \gener f(x) \le 2 c(d) b K(r) L(4 R) - \frac{K(r) L(4 R)}{c(d)} \, .
\end{equation*}
We conclude that if $2 c(d) b < 1 / c(d)$, then $\gener f(x) \le 0$ whenever $|x| \le R + r$.

\emph{Step 6.}
Suppose now that $|x| > R + r$. Recall that $R > 5 r$, and hence, in particular, $|x| > 2 r$. Therefore,
\begin{equation}\label{eq:ret:geng2}
 \gener g(x) = \int_{\R^d} g(y) \nu(y - x) dy \le \int_{B_{2 r}} \nu(y - x) dy .
\end{equation}
On the other hand, the function $h_0(x) = (R / |x|)^{d - 2}$ is a superharmonic function in $\R^d$ in the classical sense, and hence also with respect to $X_t$ (as in the remark at the beginning of the proof). Thus,
\begin{equation*}
 \gener h(x) = \gener h_0(x) + \gener (h - h_0)(x) \le \gener (h - h_0)(x) .
\end{equation*}
Furthermore, $h - h_0 = h - h_R = 0$ in a neighborhood of $x$, and $h - h_0 \le h_R - h_0$. Hence,
\begin{equation}\label{eq:ret:genh2}
 \gener h(x) \le \gener (h_R - h_0)(x) = \int_{B_R} \left(1 - \frac{R^{d - 2}}{|y|^{d - 2}}\right) \nu(y - x) dy .
\end{equation}
The estimates~\eqref{eq:ret:geng2} and~\eqref{eq:ret:genh2} imply that
\begin{equation*}
\begin{aligned}
 \gener f(x) & = b L(r) \gener g(x) + K(r) \gener h(x) \\
 & \le b L(r) \int_{B_{2 r}} \nu(y - x) dy - K(r) \int_{B_R} \left(\frac{R^{d - 2}}{|y|^{d - 2}} - 1\right) \nu(y - x) dy \\
 & \le b L(r) \nu(|x| - 2 r) |B_{2 r}| - K(r) \nu(|x| - 2 r) \int_{B_R \cap B(x, |x| - 2 r)} \left(\frac{R^{d - 2}}{|y|^{d - 2}} - 1\right) dy .
\end{aligned}
\end{equation*}
Let $x_0$ be a point on a line segment $[0, x]$ such that $|x_0| = 3 R/5$. If $|y - x_0| < R/5$, then $|y| < R$ and $|y - x| \le |y - x_0| + |x_0 - x| < R/5 + (|x| - 3 R/5) = |x| - 2 R/5 < |x| - 2 r$, and thus the set $B_R \cap B(x, |x| - 2 r)$ contains $B(x_0, R/5)$. Furthermore, $|y| < 4 R/5$ for $y \in B(x_0, R/5)$, and so
\begin{equation*}
 \int_{B_R \cap B(x, |x| - 2 r)} \left(\frac{R^{d - 2}}{|y|^{d - 2}} - 1\right) dy \ge \int_{B(x_0, R/5)} ((5/4)^{d - 2} - 1) dy = \frac{1}{c(d)} \, R^d .
\end{equation*}
It follows that 
\begin{equation*}
 \gener f(x) \le c(d) b r^d L(r) \nu(|x| - 2 r) - \frac{1}{c(d)} R^d K(r) \nu(|x| - 2 r) .
\end{equation*}
Hence, by the definition of $R$,
\begin{equation*}
 \gener f(x) \le c(d) b r^d L(r) \nu(|x| - 2 r) - \frac{1}{c(d)} \, a r^d L(r) \nu(|x| - 2 r) .
\end{equation*}
We conclude that if $c(d) b - a / c(d) < 0$, then $\gener f(x) \le 0$ whenever $|x| > R + r$.

\emph{Step 7.}
Let us summarize the conditions on $a$ and $b$: in \emph{Step~5} we required that $a < 1 / c(d)$ and $b < 1 / c(d)$, while for \emph{Step~6} we needed $b < a / c(d)$. Thus, regardless of the values of the three constants $c(d)$, it is possible to choose appropriate $a$ and $b$. With this choice, we have $\gener f(x) \le 0$ for all $x \in \R^d$. Superharmonicity of $f$ is now a direct consequence of Dynkin's formula~\eqref{eq:dynkin}: for any bounded open $D \sub \R^d$,
\begin{equation*}
 \ex^x f(X_{\tau_D}) - f(x) = \ex^x \int_0^{\tau_D} \gener f(X_t) dt \le 0 ,
\end{equation*}
as desired.
\end{proof}

\subsection{Potential kernel and Green function of a ball}

As mentioned above, the upper bound for $U(z)$ follows relatively easily from Lemma~\ref{lem:ret}.

\begin{proof}[Proof of the upper bound of Theorem~\ref{thm:pot}]
Let $T_r = \tau_{\R^d \setminus \overline{B}_r}$ be the hitting time of $\overline{B}_r$. As in Lemma~3.4 in~\cite{KX}, we have, by the strong Markov property,
\begin{equation*}
\begin{aligned}
 U(B(x, r)) & = \ex^x \int_0^\infty \ind_{B_r}(X_t) dt = \ex^x \left( \ind_{\{T_r < \infty\}} \int_{T_r}^\infty \ind_{B_r}(X_t) dt \right) \\
 & = \ex^x \left( \ind_{\{T_r < \infty\}} \int_{T_r}^\infty \ind_{B(-X_{T_r}, r)}(X_t - X_{T_r}) dt \right) = \ex^x(\ind_{\{T_r < \infty\}} U(B(-X_{T_r}, r))) .
\end{aligned}
\end{equation*}
If $T_r < \infty$, then $X(T_r) \in B_r$, and so $B(-X_{T_r}, r) \subseteq B_{2 r}$. Thus,
\begin{equation}\label{eq:kx}
 U(B(x, r)) \le \ex^x(\ind_{\{T_r < \infty\}} U(B_{2 r})) = \pr^x(T_r < \infty) U(B_{2 r}) .
\end{equation}
By Lemma~\ref{lem:ret} and estimate~\eqref{eq:pot:bound}, if $|x| \ge 2 r$, then
\begin{equation*}
 U(B(x, r)) \le c(d) \, \frac{K(r)}{K(r) + L(r)} \, \frac{1}{K(2 r) + L(2 r)} \, .
\end{equation*}
On the other hand,
\begin{equation*}
 U(B(x, r)) = \int_{B(x, r)} U(z) dz \ge |B_r| \, U(|x| + r) .
\end{equation*}
Choosing $x$ so that $|x| = 2 r$, we conclude that
\begin{equation*}
 U(3 r) \le c(d) \, \frac{K(r)}{r^d (K(r) + L(r)) (K(2 r) + L(2 r))} \le c(d) \, \frac{K(3 r)}{(3 r)^d (K(3 r) + L(3 r))^2} \, ,
\end{equation*}
which is equivalent to the desired upper bound.
\end{proof}

Before proving the lower bound, we first consider the Green function of a ball.

\begin{proof}[Proof of the upper bound of Theorem~\ref{thm:green}]
Assume first that $y = -x$. We have
\begin{equation*}
 G_{B_r}(x, -x) = 2 \int_0^\infty P^{B_r}_{2t}(x, -x) dt = 2 \int_0^\infty \left( \int_{B_r} P^{B_r}_t(x, z) P^{B_r}_t(z, -x) dz \right) dt .
\end{equation*}
Let $D = \{ z \in B_r : z \cdot x \ge 0 \}$ be a semi-ball. Then, by symmetry,
\begin{equation*}
\begin{aligned}
 G_{B_r}(x, -x) & = 4 \int_0^\infty \left( \int_D P^{B_r}_t(x, z) P^{B_r}_t(z, -x) dz \right) dt \\
 & \le 4 \int_0^\infty \left( \left(\sup_{z \in D} P^{B_r}_t(z, -x) \right) \int_{B_r} P^{B_r}_t(x, z) dz \right) dt .
\end{aligned}
\end{equation*}
Observe that $P^{B_r}_t(-x, z) \le P_t(z + x)$, and since $P_t$ is an isotropic unimodal kernel, for $z \in D$ we have $P_t(z + x) \le P_t(x)$. Therefore,
\begin{equation}
\label{eq:gd:sym}
 G_{B_r}(x, -x) \le 4 \int_0^\infty P_t(x) \pr^x(\tau_{B_r} > t) dt .
\end{equation}
By~\eqref{eq:pt:bound:u},
\begin{equation*}
 P_t(x) \le c(d) \, \frac{t K(|x|)}{|x|^d} \, .
\end{equation*}
The estimate of $\pr^x(\tau_D > t)$ is an extension of formula~(3.2) in~\cite{MR632968}, which originally asserts that $\pr^0(\tau_{B_r} > t) \le c(d) / (t (K(r) + L(r)))$: by translation invariance,
\begin{equation*}
 \sup_{z \in B_r} \pr^z(\tau_{B_r} > t) \le \pr^0(\tau_{B_{2 r}} > t) \le c(d) \, \frac{1}{t (K(2 r) + L(2 r))} \le c(d) \, \frac{1}{t (K(r) + L(r))} \, .
\end{equation*}
Choose $t_0$ so that the right-hand side is equal to $1/2$ for $t = t_0$. Using the Markov property and induction, we obtain
\begin{equation*}
 \sup_{z \in B_r} \pr^z(\tau_{B_r} > t) \le (1/2)^{\lfloor t/t_0 \rfloor} \le 2^{1 - t / t_0} \, ,
\end{equation*}
that is,
\begin{equation*}
 \sup_{z \in B_r} \pr^z(\tau_{B_r} > t) \le c(d) e^{-c(d) t (K(r) + L(r))} .
\end{equation*}
(this estimate is in fact valid for general L\'evy processes, given that the term $K(r) + L(r)$ is appropriately modified, see~\cite{MR632968}). By combining the bounds obtained above with~\eqref{eq:gd:sym} and evaluating the integral, we obtain
\begin{equation*}
 G_D(x, -x) \le c(d) \, \frac{K(|x|)}{(K(r) + L(r))^2 |x|^d} .
\end{equation*}
For general $x, y$, denoting $z = (x + y) / 2$ and $x' = x - z = z - y = (x - y) / 2$, we have
\begin{equation*}
\begin{aligned}
 G_{B_r}(x, y) & \le G_{B(z, 2r)}(x, y) = G_{B_{2 r}}(x', -x') \\
 & \le c(d) \, \frac{K(|x'|)}{(K(2 r) + L(2 r))^2 |x'|^d} \le c(d) \, \frac{K(|x - y|)}{(K(r) + L(r))^2 |x - y|^d} \, ,
\end{aligned}
\end{equation*}
as desired.
\end{proof}

\begin{proof}[Proof of the lower bound of Theorem~\ref{thm:green}]
Let $p = \min(|x - y|, r - |x|) / 2$ and $q = \min(|x - y|, r - |y|) / 2$. We have
\begin{equation*}
\begin{aligned}
 G_{B_r}(x, y) & = \ex^x G_{B_r}(X(\tau_{B(x, p)}), y) \\
 & \ge \int_{\R^d \setminus B(x, p)} P_{B(x, p)}(x,z) G_{B_r}(z, y) dz \ge \int_{B(y, q)} P_{B(x, p)}(x,z) G_{B(y, q)}(z, y) dz .
\end{aligned}
\end{equation*}
By translation invariance and unimodality of $P_{B_p}$ (in the sense of Corollary~\ref{cor:pb:unimodal}), we have $P_{B(x, p)}(x, z) \ge P_{B_p}(0, y - x)$ for $z \in B(y, q) \cap \overline{B}(x, |x - y|)$. Hence,
\begin{equation*}
\begin{aligned}
 G_{B_r}(x, y) & \ge P_{B_p}(0, y - x) \int_{B(y, q) \cap \overline{B}(x, |x - y|)} G_{B(y, q)}(z, y) dz \\
 & = P_{B_p}(0, y - x) \int_{B_q \cap \overline{B}(x - y, |x - y|)} G_{B_q}(z, 0) dz .
\end{aligned}
\end{equation*}
Since $G_{B_q}(z, 0)$ is isotropic, and $B_q$ can be covered by a constant number $c(d)$ of rotations of $B_q \cap \overline{B}(x - y, |x - y|)$ (this is the same argument as in Lemma~\ref{lem:pb:bound}), we have
\begin{equation*}
 G_{B_r}(x, y) \ge \frac{1}{c(d)} \, P_{B_p}(0, x - y) \int_{B_q} G_{B_q}(z, 0) dz = \frac{1}{c(d)} \, P_{B_p}(0, x - y) \ex^0 \tau_{B_q} .
\end{equation*}
By the estimate~\eqref{eq:pruitt} and Lemma~\ref{lem:pb:bound},
\begin{equation*}
 G_{B_r}(x, y) \ge \frac{1}{c(d)} \, \frac{\nu(x - y)}{K(p) + L(p)} \, \frac{1}{K(q) + L(q)} \, ,
\end{equation*}
which is equivalent to the desired result.
\end{proof}

\begin{proof}[Proof of the lower bound of Theorem~\ref{thm:pot}]
The result follows directly from the inequality $U(z) \ge G_{B_r}(0, z)$, with $r = 2 |z|$, and the lower bound of Theorem~\ref{thm:green}.
\end{proof}

\subsection{Probability of hitting an arbitrary ball}

With Lemma~\ref{lem:ret} and Theorem~\ref{thm:pot} at hand, we are in position to prove Theorem~\ref{thm:ret}.

\begin{proof}[Proof of the upper bound of Theorem~\ref{thm:ret}]
Lemma~\ref{lem:ret} asserts the result when $|x| \le 3 r$. Thus, we restrict our attention to the case $|x| > 3 r$. By Lemma~3.5 in~\cite{KX},
\begin{equation*}
 \pr^x(T_r < \infty) \le \frac{U(B(x, 2 r))}{U(B_r)}
\end{equation*}
(this is fully analogous to~\eqref{eq:kx}: one uses the fact that $B(-X_{T_r}, 2 r) \supseteq B_r$). Since
\begin{equation*}
 U(B(x, 2 r)) = \int_{B(x, 2 r)} U(z) dz \le |B(x, 2 r)| \, U(|x| - 2 r) \le c(d) r^d U(|x| / 3) ,
\end{equation*}
we have, by the estimate~\eqref{eq:pot:bound},
\begin{equation*}
 \pr^x(T_r < \infty) \le c(d) r^d (K(r) + L(r)) U(|x| / 3) .
\end{equation*}
It remains to use the upper bound of Theorem~\ref{thm:pot}.
\end{proof}

\begin{proof}[Proof of the lower bound of Theorem~\ref{thm:ret}]
This is nearly identical to the proof of the upper bound: using~\eqref{eq:kx}, the estimate
\begin{equation*}
 U(B(x, r)) = \int_{B(x, r)} U(z) dz \ge |B(x, r)| \, U(|x| + r) \ge \frac{1}{c(d)} \, r^d U(2 |x|) ,
\end{equation*}
and the estimate~\eqref{eq:pot:bound}, we obtain
\begin{equation*}
 \pr^x(T_r < \infty) \ge \frac{U(B(x, r))}{U(B_{2 r})} \ge \frac{1}{c(d)} \, r^d (K(r) + L(r)) U(2 |x|) .
\end{equation*}
The desired estimate follows from the lower bound of Theorem~\ref{thm:pot}.
\end{proof}

\subsection{Harmonic functions}

By averaging the kernels $P_{B_r}(0, dz)$ with respect to $r$, one can easily prove the supremum estimate stated in Theorem~\ref{thm:sup}. In fact, we prove the following more refined result.

\begin{prop}\label{prop:sup}
Let $X_t$ be an isotropic unimodal L\'evy process in $\R^d$, and suppose that $0 \le q < r$. Then there is a radial kernel function $\bar{P}_{q,r}(z)$, a constant $\cbarp(q, r) > 0$, and a probability measure $\mu_{q,r}$ on $[q, r]$ with the following properties:
\begin{itemize}
\item if $f$ is a regular harmonic function in $B_r$, then
\begin{equation*}
 f(0) = \int_{\R^d \setminus B_q} f(z) \bar{P}_{q,r}(z) dz ;
\end{equation*}
\item for any $A \sub \R^d$,
\begin{equation*}
 \bar{P}_{q,r}(A) = \int_{[q, r]} P_{B_s}(0, A) \mu_{q,r}(ds) ;
\end{equation*}
\item $\bar{P}_{q,r}(z) = 0$ for $z \in B_q$;
\item $\bar{P}_{q,r}(z) = \cbarp(q, r)$ for $z \in B_r \setminus B_q$;
\item $0 \le \bar{P}_{q,r}(z) \le \cbarp(q, r)$ for all $z \in \R^d$;
\item the profile function of $\bar{P}_{q,r}(z)$ is non-increasing in $(r, \infty)$, and
\begin{equation}\label{eq:barp:repr}
\begin{aligned}
 \bar{P}_{q,r}(z) & = \int_{[q, r]} P_{B_s}(0, z) \mu_{q,r}(ds) \\
 & = \int_{[q, r]} \int_{B_s} \nu(|z - y|) G_{B_s}(0, dy) \mu_{q,r}(ds)
\end{aligned}
\end{equation}
for $z \in \R^d \setminus \overline{B}(0, r)$.
\end{itemize}
In particular,
\begin{equation}\label{eq:barp:est}
 \bar{P}_{q,r}(z) \le P_{B_r}(0, z)
\end{equation}
for $z \in \R^d \setminus B_r$. Furthermore, with $\bar{q} = (q + r) / 2$,
\begin{equation}\label{eq:barp:c1}
 \cbarp(q, r) \le \frac{P_{B_{\bar{q}}}(0, B_r \setminus B_{\bar{q}})}{|B_r \setminus B_{\bar{q}}|} \le c(d, q/r) \, \frac{K(r)}{r^d (K(r) + L(r))} \, ,
\end{equation}
and if $|z| = r$, then
\begin{equation}\label{eq:barp:c2}
 \cbarp(q, r) \ge \frac{1}{c(d, q/r)} \, P_{B_{\bar{q}}}(0, z) \ge \frac{1}{c(d, q/r)} \, \frac{\nu(r)}{K(r) + L(r)} \, .
\end{equation}
\end{prop}

\begin{proof}
The existence of $\bar{P}_{q,r}$, $\cbarp(q, r)$ and $\mu_{q,r}$ is a rather general result, which is likely well known, but difficult to find in the literature. For completeness, we provide a proof for $r = 1$ in Appendix~\ref{sec:app:sup}, see Lemma~\ref{lem:reg}; extension to the case of general $r$ is automatic. It remains to derive~\eqref{eq:barp:est}, \eqref{eq:barp:c1} and~\eqref{eq:barp:c2}.

The estimate~\eqref{eq:barp:est} is a consequence of~\eqref{eq:barp:repr} and the fact that $P_{B_s}(0, z)$ is a non-decreasing function of $s \in [q, r]$. To prove the first inequality in~\eqref{eq:barp:c1}, observe that, with $\bar{q} = (q + r) / 2$, we have
\begin{equation*}
\begin{aligned}
 \cbarp(q, r) |B_r \setminus B_{\bar{q}}| & = \int_{B_r \setminus B_{\bar{q}}} \bar{P}_{q,r}(z) dz \\
 & = \int_{[q, \bar{q}]} P_{B_s}(0, B_r \setminus B_{\bar{q}}) \mu_{q,r}(ds) + \int_{(\bar{q}, r]} P_{B_s}(0, B_r \setminus B_s) \mu_{q,r}(ds) \\
 & \le P_{B_{\bar{q}}}(0, B_r \setminus B_{\bar{q}}) ;
\end{aligned}
\end{equation*}
in the last inequality we used the fact that $P_{B_s}(0, B_r \setminus B_{\bar{q}})$ is a non-decreasing function of $s \in [q, \bar{q}]$ and $P_{B_s}(0, B_r \setminus B_s)$ is a non-increasing function of $s \in [\bar{q}, r]$.

Obviously, $P_{B_{\bar{q}}}(0, B_r \setminus B_{\bar{q}}) \le 1$, which in most cases gives a satisfactory bound of $\cbarp(q, r)$. However, if the decay of $\nu$ at infinity is slow, a better bound can be found using Dynkin's formula~\eqref{eq:dynkin}, as in Lemma~3.1 in~\cite{BKK}. Let $f$ be a smooth function which takes values in $[0, 1]$, which is equal to $1$ in $B_r \setminus B_{\bar{q}}$, and equal to $0$ outside $B_{2r} \setminus \{0\}$, and such that $\|f''\|_\infty < c(d, q/r) / r^2$ (with the notation of Lemma~\ref{lem:generator}). By Lemma~\ref{lem:generator}, $\gener f(x) \le c(d, q/r) K(r)$. Hence, by~\eqref{eq:dynkin}
\begin{equation*}
\begin{aligned}
 P_{B_{\bar{q}}}(0, B_r \setminus B_{\bar{q}}) & \le \ex^0 f(X(\tau_{B_{\bar{q}}})) \\
 & = f(0) + \ex^0 \int_0^{\tau_{B_{\bar{q}}}} \gener f(X_t) dt \le c(d, q/r) K(r) \ex^0 \tau_{B_{\bar{q}}} .
\end{aligned}
\end{equation*}
By Pruitt's estimate~\eqref{eq:pruitt},
\begin{equation*}
 P_{B_{\bar{q}}}(0, B_r \setminus B_{\bar{q}}) \le c(d, q/r) \, \frac{K(r)}{K(\bar{q}) + L(\bar{q})} \le c(d, q/r) \, \frac{K(r)}{K(r) + L(r)} \, .
\end{equation*}

For the proof of~\eqref{eq:barp:c2}, again let $\bar{q} = (q + r) / 2$, let $|z| = r$, and consider two scenarios. If $\mu_{q,r}([q, \bar{q}]) \ge 1/2$, then
\begin{equation*}
\begin{aligned}
 \cbarp(q, r) |B_r \setminus B_q| & = \int_{B_r \setminus B_q} \bar{P}_{q,r}(z) \ge \int_{[q,\bar{q}]} P_{B_s}(0, B_r \setminus B_s) \mu_{q,r}(ds) \\
 & \ge \frac{P_{B_{\bar{q}}}(0, B_r \setminus B_{\bar{q}})}{2} \ge \frac{P_{B_{\bar{q}}}(0, z) |B_r \setminus B_{\bar{q}}|}{2} \, ;
\end{aligned}
\end{equation*}
here we used the fact that $P_{B_s}(0, B_r \setminus B_s)$ is a non-increasing function of $s \in [q, \bar{q}]$, and then unimodality of $P_{B_{\bar{q}}}(0, z)$ (in the sense of Corollary~\ref{cor:pb:unimodal}). On the other hand, if $\mu_{q,r}([\bar{q}, r]) \ge 1/2$, then, since $P_{B_s}(0, z)$ is a non-decreasing function of $s \in [\bar{q}, r]$, we obtain
\begin{equation*}
 \cbarp(q, r) = \bar{P}_{q,r}(z) \ge \int_{[\bar{q},r]} P_{B_s}(0, z) \mu_{q,r}(ds) \ge \frac{P_{B_{\bar{q}}}(0, z)}{2} \, .
\end{equation*}
In either case, we obtain the first part of~\eqref{eq:barp:c2}. The other part is just the lower bound of Lemma~\ref{lem:pb:bound}.
\end{proof}

\begin{proof}[Proof of Theorem~\ref{thm:sup}]
Due to translation invariance, with no loss of generality we may assume that $x_0 = 0$. Proposition~\ref{prop:sup} gives the upper bound of Theorem~\ref{thm:sup} for $x = 0$. In order to extend this to $x \in B_q$, one can use a sweeping argument. However, a shorter proof can be obtained using the following trick. As remarked above, Proposition~\ref{prop:pb:unimodal} and Corollary~\ref{prop:pb:unimodal}, and thus also Proposition~\ref{prop:sup}, extend to the case when $X_0$ is distributed uniformly in $B_p$, given that $0 < p \le q < r$. Denote the resulting kernel and constant by $\bar{P}_{p,q,r}$ and $\cbarp(p, q, r)$ (instead of $\bar{P}_{q,r}$ and $\cbarp(q, r)$). Hence, if $f$ is a regular harmonic function in $B_r$, then
\begin{equation*}
 \frac{1}{|B_p|} \int_{B_p} f(y) dy = \int_{\R^d \setminus {B_q}} f(z) \bar{P}_{p,q,r}(z) dz .
\end{equation*}
Furthermore, by the same argument as above (with the smooth function equal to zero outside $B_{2r} \setminus B_q$ instead of $B_{2r} \setminus \{0\}$),
\begin{equation*}
 \cbarp(p, q, r) \le c(d, p/r, q/r) \, \frac{K(r)}{r^d (K(r) + L(r))} .
\end{equation*}
Let $\bar{q} = (q + r)/2$. With the notation of Proposition~\ref{prop:sup}, for all $x \in \overline{B}_q$ we have
\begin{equation*}
 f(x) = \int_{\R^d} f(z) \bar{P}_{0,r - q}(z - x) dz \le \cbarp(0, r - q) \int_{\R^d} f(z) dz .
\end{equation*}
Furthermore,
\begin{equation*}
\begin{aligned}
 \int_{B_q} f(y) dy & = |B_q| \int_{\R^d \setminus B_{\bar{q}}} f(z) \bar{P}_{q,\bar{q},r}(z) dz \\
 & \le {|B_q|} \cbarp(q,\bar{q},r) \int_{\R^d \setminus B_{\bar{q}}} f(z) dz .
\end{aligned}
\end{equation*}
It follows that
\begin{equation*}
\begin{aligned}
 f(x) & \le \cbarp(0, {r - q}) \left( \int_{\R^d \setminus B_q} f(z) dz + |B_q| \cbarp(q, \bar{q}, r) \int_{\R^d \setminus B_{\bar{q}}} f(z) dz \right) \\
 & \le \cbarp(0, {r - q}) (1 + |B_q| \cbarp(q, \bar{q}, r)) \int_{\R^d \setminus B_q} f(z) dz
\end{aligned}
\end{equation*}
for all $x \in {\overline{B}_q}$. Since $\cbarp(q, \bar{q}, r) \le c(d, q/r) r^{-d}$, the constant in the right-hand side does not exceed $c(d, q/r) \cbarp(0, {r - q})$. This proves the upper bound. The lower bound is simply an application of Proposition~\ref{prop:sup}.
\end{proof}

\begin{proof}[Proof of Theorem~\ref{thm:reg}]
We use the notation of Proposition~\ref{prop:sup}, with $q = 0$. If $f$ is a bounded function which is a regular harmonic function in $B_{2 r}$, then $f = f * \bar{P}_{0,r}$ in $B_r$, which already gives continuity of $f$. This can be extended as follows. Let $\kappa$ be a non-negative smooth radial function which takes values in $[0, 1]$, which is equal to $1$ in $B(0, 3 r / 2)$, and which is equal to $0$ in $\R^d \setminus B(0, 2 r)$. Define $\pi_r(z) = \bar{P}_{0,r}(z) \kappa(z)$, $\Pi_r(z) = \bar{P}_{0,r}(z) (1 - \kappa(z))$. If $f$ is a bounded function which is a regular harmonic function in $B_{(k+1) r}$, then, by induction,
\begin{equation*}
\begin{aligned}
 f & = (\Pi_r + \pi_r * \Pi_r + \pi_r^{*2} * \Pi_r + \ldots + \pi_r^{*(k-1)} * \Pi_r + \pi_r^{*k}) * f \\
 & = (\delta_0 + \pi_r + \pi_r^{*2} + \ldots + \pi_r^{*(k-1)}) * \Pi_r * f + \pi_r^{*k} * f .
\end{aligned}
\end{equation*}
in $B_r$. Therefore, at least informally, $f$ is at least as smooth in $B_r$, as $\Pi_r * f$ and $\pi_r^{*k} * f$ are on $\R^d$.

When $d = 1$, then the distributional derivative of $\pi_r$ is a finite measure (because $\pi_r$ is unimodal). It follows that the Fourier transform of $\pi_r$ is bounded by $C(r) \min(1, |\xi|^{-1})$. To obtain a similar estimate for $d \ge 2$, simply note that $\pi_r$ is a bounded, integrable radial function; hence its Fourier transform is bounded by $C(r) \min(1, |\xi|^{-(d-1)/2})$ (see Proposition~A.4 in~\cite{MR0338688}). It follows that for any $d \ge 1$, for $k$ sufficiently large, $\pi_r^{*k}$ has continuous partial derivatives of order up to $N$. Since $\pi_r^{*k}$ has compact support, $\pi_r^{*k} * f$ has continuous partial derivatives of order up to $N$. On the other hand,
\begin{equation*}
 \Pi_r(z) = \int_{[0,r]} \int_{B(0, t)} \nu(|z - y|) (1 - \kappa(z)) G_{B(0, t)}(0, dy) \mu(dt) .
\end{equation*}
By the assumption, the functions $g_y(z) = \nu(|z - y|) (1 - \kappa(z))$ have absolutely integrable partial derivatives of order up to $N$, with $L^1(\R^d)$ norms bounded uniformly in $y \in B(0, r)$. Hence $\Pi_r(z)$ has the same property. Consequently, $\Pi_r * f$ has continuous partial derivatives of order up to $N$. It follows that $f$ has continuous derivatives of order up to $N$ in $B_r$. Furthermore, these derivatives are bounded by $\|f\|_\infty$ multiplied by a constant which depends only on $\pi_r$ and $\Pi_r$.

The desired result follows by translation invariance: instead of $B_r$, consider $B(x, r)$, where $x$ is any point in $D$ such that $\dist(x, \R^d \setminus D) \ge \delta$, and where $r$ is small enough, so that $(k + 1) r < \delta$.
\end{proof}

%
%

\section{Boundary Harnack inequality}
\label{sec:bhi}

In this section we apply the boundary Harnack inequality of~\cite{BKK} to unimodal L\'{e}vy processes and prove Theorem~\ref{thm:bhi}. We also apply the recent result of~\cite{JK15} to prove Theorem~\ref{thm:martin}. Finally, we find estimates of the Green function of the half-space.

\subsection{Conditions for the boundary Harnack inequality}

Application of the results of~\cite{BKK} and~\cite{JK15} requires verification of a number of conditions imposed on the process $X_t$ in~\cite{BKK}. Below we briefly discuss these conditions.

Assumption~A in~\cite{BKK} requires $X_t$ to be a Hunt process which admits a dual Hunt process $\hat{X}_t$, and both are required to have Feller and strong Feller property. Any L\'evy process has Feller property (because convolution with a measure maps $C_0$ functions into $C_0$ functions), and if it has transition densities, it also has strong Feller property (because convolution with an integrable function maps bounded measureable functions to continuous ones). In this case it also has a dual Hunt process $2 X_0 - X_t$ (duality in the sense of potential theory requires that $\lambda$-potential densities exist). In particular, if $X_t$ is an isotropic unimodal L\'evy process which is not a compound Poisson process (which will be tacitly assumed throughout this section), then it satisfies Assumption~A of~\cite{BKK}.

Smooth, compactly supported functions belong to the domain of the Feller generator $\gener$ of $X_t$, so Assumption~B in~\cite{BKK} is satisfied. The constant in the boundary Harnack inequality depends on the quantity
\begin{equation}\label{eq:rho:est}
 \cro(F, D) = \inf_{f} \sup_{x \in \R^d} \gener f(x) ,
\end{equation}
introduced in Assumption~B. Here $F$ is compact, $D$ is open, $F \subseteq D$, and the infimum is taken over all smooth, compactly supported $f$ such that $0 \le f \le 1$, $f = 0$ in $F$ and $f = 1$ in $\R^d \setminus D$. A sufficient estimate of $\cro(F, D)$ is given by Lemma~\ref{lem:generator}: if $D$ is contained in a ball of radius $R$ and $r$ is the distance between $F$ and $\R^d \setminus D$, then, choosing $f$ such that $\|f''\|_\infty \le c(d) / r^2$, we get
\begin{equation*}
 \cro(F, D) \le c(d) (R / r)^2 K(R) .
\end{equation*}

Assumption~C in~\cite{BKK} requires that the density of the L\'evy kernel is comparable to a constant on appropriate balls. For the isotropic unimodal L\'evy processes, this reduces to the following condition on the radial profile $\nu$ of the density of the L\'evy measure: if $0 < r < R < R_\infty$ (here $R_\infty \in (0, \infty]$ is a localisation radius), then
\begin{equation}\label{eq:levy:condition}
  \nu(s - r) \le \cnu(r, R) \nu(s + r)
\end{equation}
for all $s > R$. In order to have a scale-invariant version of the boundary Harnack inequality, we will need to assume that $\cnu$ is scale-invariant too. For the existence of boundary limits, we will also need the condition $\lim_{r \to 0^+} \cnu(r, R) = 1$ for all $R > 0$.

The last condition in~\cite{BKK}, Assumption~D, is a rough upper bound for Green functions of balls, away from the diagonal. For isotropic unimodal L\'evy processes, such an estimate is provided by Theorem~\ref{thm:green}: if $0 < r < s < R$, we have
\begin{equation}\label{eq:green:est}
\begin{aligned}
 \cgreen(r, s, R) & = \sup_{x \in B_r} \sup_{y \in \R^d \setminus B_s} G_{B_R}(x, y) \\
 & \le c(d) \, \frac{K(s - r)}{(s - r)^d (K(R) + L(R))^2} \le c(d) \, \frac{R^2}{(s - r)^{d + 2} (K(R) + L(R))} \, .
\end{aligned}
\end{equation}
Another constant that is involved in the boundary Harnack inequality is related to the first exit time from a ball: by translation invariance and Pruitt's estimate~\eqref{eq:pruitt},
\begin{equation}\label{eq:pruittctau}
\begin{aligned}
 \ctau(r) = \sup_{x \in B_r} \ex^x \tau_{B_r} & \le \sup_{x \in B_r} \ex^x \tau_{B(x, 2 r)} \\
 & \le \ex^0 \tau_{B_{2 r}} \le \frac{c(d)}{K(2 r) + L(2 r)} \le \frac{c(d)}{K(r) + L(r)} \, .
\end{aligned}
\end{equation}

\subsection{Boundary Harnack inequality}

The following theorem is a reformulation of the main result of~\cite{BKK}, adapted to the setting of isotropic unimodal L\'evy processes. We use a version stated in~\cite{JK15}. We use the above notation instead of that introduced in~\cite{BKK}.

\begin{thm}[{Lemma~3.2 and Theorems~3.4 and~3.5 in~\cite{BKK}, see Theorem~1 in~\cite{JK15}}]
\label{thm:bhi:full}
Let $X_t$ be an isotropic unimodal L\'evy process which is not a compound Poisson process, and let $R_\infty \in (0, \infty]$. Suppose that the L\'evy measure of $X_t$ is positive and it satisfies condition~\eqref{eq:levy:condition} whenever $0 < r < R < R_\infty$, $s > R$. Let $x_0 \in \R^d$, $0 < r < R < R_\infty$, and let $r(t) = (1 - t) r + t R$. Suppose that a non-negative function $f$ is a regular harmonic function in $D \cap B(x_0, R)$, which is equal to zero in $B(x_0, R) \setminus D$. Then
\begin{equation*}
 f(x) \approx \cbhi(r, R) \ex^x \tau_{D \cap B(x_0, r(1/3))} \int_{\R^d \setminus B(x_0, r(2/3))} f(y) \nu(|y - x_0|) dy
\end{equation*}
for $x \in D \cap B(x_0, r)$, where
\begin{equation*}
 \cbhi(r, R) = c(d) \left(\frac{R}{r}\right)^{\!\! d} \left(\frac{R}{R - r}\right)^{\!\! 4} (\cnus(r, R))^3 \frac{K(2 R)}{R^d \nu(2 R)} \, ,
\end{equation*}
and
\begin{equation*}
 \cnus(r, R) = \max \bigl( \cnu(r(1/3), r(2/3)), \cnu(r(7/9), r(8/9)), \cnu(R, r(4/3)) \bigr) .
\end{equation*}
\end{thm}

The constant in the original statement is given by a more complicated expression:
\begin{equation*}
\begin{aligned}
 \cbhit(r, R) & = \cnu(r_2, r_3) + 2 \cro(\overline{B}(0, r_3) \setminus B(x_0, r_2), B(x_0, r_8) \setminus \overline{B}(0, r)) \times \\
 & \hspace*{8em} \times \left( \cgreen(r_3, r_4, R) + \frac{\ctau(R) (\cnu(r_4, r_5))^2}{|B(x_0, r_4)|} \right) \times \\
 & \hspace*{8em} \times \left( \frac{\cro(\overline{B}(x_0, r_5), B(x_0, R))}{\nu(r_7)} + \cnu(R, r_7) |B(x_0, R)| \right) ,
\end{aligned}
\end{equation*}
where $r_2 = r(1/3)$, $r_3 = r(2/3)$ and, for example, $r_4 = r(7/9)$, $r_5 = r(8/9)$, $r_7 = r(4/3)$, $r_8 = r(5/3)$. The above constant is even better than the one given in the statement of Theorem~\ref{thm:bhi:full}; indeed, by~\eqref{eq:rho:est}, \eqref{eq:green:est} and~\eqref{eq:pruittctau}, we have
\begin{equation*}
\begin{aligned}
 \cbhit(r, R) & \le \cnu(r_2, r_3) + c(d) \left(\frac{R}{R - r}\right)^{\!\! 2} K(2 R) \times \\
 & \hspace*{8em} \times \left( \frac{1}{r^d (K(R) + L(R))} + \frac{(\cnu(r_4, r_5))^2}{r^d (K(R) + L(R))} \right) \times \\
 & \hspace*{8em} \times \left( \smash{\left(\frac{R}{R - r}\right)^{\!\! 2}} \frac{K(2 R)}{\nu(2 R)} + \cnu(R, r_7) R^d \right) ,
\end{aligned}
\end{equation*}
and hence, by elementary manipulations,
\begin{equation*}
\begin{aligned}
 \cbhit(r, R) & \le \cnus(r, R) + 2 c(d) \left(\frac{R}{r}\right)^{\!\! d} \left(\frac{R}{R - r}\right)^{\!\! 4} (\cnus(r, R))^3 \left( \frac{K(2 R)}{R^d \nu(2 R)} + 1 \right) \\
 & \le \cbhi(r, R) .
\end{aligned}
\end{equation*}

We say that a constant $C(r, R)$ is \emph{scale-invariant} on $(0, R_\infty)$ if for each fixed ratio $a = r / R \in (0, 1)$, $C(a R, R)$ is bounded in $R \in (0, R_\infty)$.

\begin{proof}[Proof of Theorem~\ref{thm:bhi}]
We need to prove that, under the assumptions of the theorem, the constant $\cbhi$ of Theorem~\ref{thm:bhi:full} is scale-invariant on $(0, R_\infty)$. This is clearly the case, provided that $\cnu$ is scale-invariant on $(0, 2 R_\infty)$ (and hence $\cnus$ is scale-invariant on $(0, R_\infty)$), and that $K(r) / (r^d \nu(r))$ is bounded for $r < 2 R_\infty$.

We claim that the former condition follows from the assumptions~\eqref{eq:scaling} and~\eqref{eq:harnack} of the theorem. Indeed, suppose that $0 < R < 2 R_\infty$, $a \in (0, 1)$ and $s > R$. If $s < 2 (1 - a) R_\infty$, then $s + a R < 2 R_\infty$, and so, by~\eqref{eq:scaling},
\begin{equation*}
 \frac{\nu(s - a R)}{\nu(s + a R)} \le M \left(\frac{s + a R}{s - a R}\right)^{d + \alpha} \le M \left(\frac{1 + a}{1 - a}\right)^{d + \alpha} .
\end{equation*}
Suppose now that $R_\infty < \infty$ and $s > 2 (1 - a) R_\infty$ (and still $s > R$). If $R \ge (1 - a) R_\infty$, then $s - a R > (1 - a) R \ge (1 - a)^2 R_\infty$; otherwise, $s - a R > 2 (1 - a) R_\infty - R > (1 - a) R_\infty$. Thus, in either case $s - a R > (1 - a)^2 R_\infty$, and so
\begin{equation*}
 \frac{\nu(s - a R)}{\nu(s + a R)} \le \sup_{t \ge (1 - a)^2 R_\infty} \frac{\nu(t)}{\nu(t + 2 a R)} \le \sup_{t \ge (1 - a)^2 R_\infty} \frac{\nu(t)}{\nu(t + 2 R_\infty)} ,
\end{equation*}
which is finite by~\eqref{eq:harnack}. This proves that $\nu(s - a R) / \nu(s + a R)$ is bounded from above uniformly in $R \in (0, 2 R_\infty)$ and $s > R$, and so $\cnu$ is scale-invariant on $(0, 2 R_\infty)$, as desired.

Finally, as discussed in Remark~\ref{rem:orv}, $K(r) / (r^d \nu(r))$ is indeed bounded when $\sigma^2 = 0$ and~\eqref{eq:scaling} holds with $\alpha < 2$ when $0 < r_1 < r_2 < R_\infty$. These are again the assumptions of the theorem.
\end{proof}

\subsection{Boundary limits of ratios of harmonic functions}

\begin{proof}[Proof of Theorem~\ref{thm:martin}]
In order to prove Theorem~\ref{thm:martin}, it suffices to verify the conditions of Theorem~3.1 in~\cite{JK15}. These are the following:
\begin{itemize}
\item $X_t$ satisfies Assumptions~A through~D for the boundary Harnack inequality of~\cite{BKK};
\item $\lim_{r \to 0^+} \cnu(r, R) = 1$ for all $R > 0$;
\item $\cnu$ is scale-invariant on $(0, R_\infty)$;
\item $L(r) / L(2r)$ is bounded on $(0, R_\infty/2)$;
\item $\cbhi$ is scale-invariant on $(0, R_\infty)$.
\end{itemize}
The first one is satisfied, as discussed earlier in this section. Next two conditions are the assumptions of Theorem~\ref{thm:martin}, while the last one follows from Theorem~\ref{thm:bhi}. Therefore, it suffices to verify the fourth condition. This is quite simple: we have, by~\eqref{eq:scaling},
\begin{equation*}
\begin{aligned}
 L(r) - L(2 r) & = \int_{B_{2 r} \setminus B_r} \nu(z) dz \le 2^{d + \alpha} M \int_{B_{2 r} \setminus B_r} \nu(2 z) dz \\
 & = 2^\alpha M \int_{B_{4 r} \setminus B_{2 r}} \nu(y) dy = 2^\alpha M (L(2 r) - L(4 r)) ,
\end{aligned}
\end{equation*}
and hence $L(r) \le (1 + 2^\alpha M) L(2 r)$.
\end{proof}

\subsection{Boundary estimate of the Green function}

The application of the boundary Harnack inequality to the estimates of the Green function in smooth domains is standard, see~\cite{MR1750907,MR1654824,MR2928332,MR3206464}. For simplicity, we restrict our attention to the case of the half-space, where finding the explicit decay rate of harmonic functions requires no further effort.

Let $H = \{ x \in \R^d : x_d > 0 \}$ be a half-space in $\R^d$, and let $e_d = (0, 0, \ldots, 0, 1) \in \R^d$. It is a well-known result in the fluctuation theory of (one-dimensional) L\'evy processes that there is a function $h(s)$ which is equal to zero in $(-\infty, 0]$, positive and non-decreasing on $(0, \infty)$, and such that $h(x_d)$ is a regular harmonic function in every bounded subset of $H$. Furthermore, $h(s)$ is comparable with $c(d) / \sqrt{K(s) + L(s)}$ for $s > 0$; see Proposition~2.4 in~\cite{BGR1}.

\begin{proof}[Proof of Theorem~\ref{thm:halfspace}]
Denote $r = |x - y|$ and $z = (x + y) / 2$. Suppose first that $r \le \min(x_d, y_d)$. In this case $r \le z_d$, and thus $B(z, r) \subseteq H$. Thus,
\begin{equation*}
 G_{B(z, r)}(x, y) \le G_H(x, y) \le U(y - x) .
\end{equation*}
By Theorems~\ref{thm:pot} and~\ref{thm:green},
\begin{equation*}
 \frac{1}{c(d)} \, \frac{\nu(r)}{(K(r/2) + L(r/2))^2} \le G_H(x, y) \le c(d) \, \frac{K(r)}{r^d (K(r) + L(r))^2} \, .
\end{equation*}
Since $\nu(r) \ge r^{-d} K(r) / c(d, \alpha, M)$, we have
\begin{equation}\label{eq:halfspace:bulk}
 G_H(x, y) \approx c(d, \alpha, M) \, \frac{K(r)}{r^d (K(r) + L(r))^2} \, .
\end{equation}
For general $x, y \in H$, with no loss of generality we may assume that $x_d \ge y_d$. Define $\tilde{x} = x + r e_d$, $\tilde{y} = y + r e_d$, $B_x = B(x + r e_d/2, r)$ and $B_y = B(y + r e_d/2, r)$. Since $x_d \ge y_d$, we have $|(x + r e_d/2) - y| > |x - y| = r$, and hence $y \notin \overline{B}_x$. It follows that $G_H(w, y)$ is a regular harmonic function of $w \in B_x \cap H$. Since also $h(w_d)$ is a regular harmonic function in $B_x \cap H$, we have, by the boundary Harnack inequality (Theorem~\ref{thm:bhi}),
\begin{equation*}
 \frac{G_H(x, y)}{G_H(\tilde{x}, y)} \approx c(d, \alpha, M) \, \frac{h(x_d)}{h(x_d + r)} \, .
\end{equation*}
In a similar manner, $|\tilde{x} - (y + r e_d/2)| = |x + r e_d/2 - y| > |x - y| = R$, and hence $\tilde{x} \notin \overline{B}_y$. Since $G_H(\tilde{x}, w)$ and $h(w_d)$ are regular harmonic functions of $w \in B_y \cap H$, we have, by the boundary Harnack inequality,
\begin{equation*}
 \frac{G_H(\tilde{x}, y)}{G_H(\tilde{x}, \tilde{y})} \approx c(d, \alpha, M) \, \frac{h(y_d)}{h(y_d + r)} \, .
\end{equation*}
Therefore,
\begin{equation*}
 G_H(x, y) \approx c(d, \alpha, M) \, \frac{h(x_d)}{h(x_d + r)} \, \frac{h(y_d)}{h(y_d + r)} \, G_H(\tilde{x}, \tilde{y}) \, .
\end{equation*}
However, $|\tilde{x} - \tilde{y}| = r$ and $r \le \min(\tilde{x}_d, \tilde{y}_d)$, so estimate~\eqref{eq:halfspace:bulk} applies to $G_H(\tilde{x}, \tilde{y})$. Thus,
\begin{equation*}
 G_H(x, y) \approx c(d, \alpha, M) \, \frac{h(x_d)}{h(x_d + r)} \, \frac{h(y_d)}{h(y_d + r)} \, \frac{K(r)}{r^d (K(r) + L(r))^2} \, .
\end{equation*}
It remains to recall that $h(s) \approx c(d) / \sqrt{K(s) + L(s)}$ for $s > 0$.
\end{proof}

%
%

\section{Examples}
\label{sec:ex}

\begin{exmp}
\label{ex:slow}
Consider a pure-jump isotropic unimodal L\'{e}vy process in $\R^d$ with the L\'{e}vy density $\nu(x) \approx C |x|^{-d} (1 + |x|)^{-\alpha}$ for some $\alpha \in (0, 2)$ and $C > 0$ (here and below we write $f \approx C g$ for $C^{-1} g \le f \le C g$). The class of processes considered here includes geometric stable processes, that is, L\'evy processes with characteristic exponent of the form $\Psi(z) = \log(1 + |z|^\alpha)$, where $\alpha \in (0, 2)$ (see Theorems~3.4 and~3.5 in~\cite{MR2240700}).

Let $l(r) = \log(e + 1/r)$, so that $l(r)$ is comparable with $\log(1/r)$ for small $r$ and $l(r)$ is approximately $1$ for large $r$. It is easy to see that
\begin{equation*}
\begin{aligned}
 K(r) & \approx c(d, \alpha, C) (1 + r)^{-\alpha} , & \qquad L(r) & \approx c(d, \alpha, C) (1 + r)^{-\alpha} l(r) .
\end{aligned}
\end{equation*}
By Theorems~\ref{thm:ret} and~\ref{thm:pot}, if $d \ge 3$, then
\begin{equation*}
 \pr^x(T_r < \infty) \approx c(d, \alpha, C) \, \frac{r^d (1 + r)^{-\alpha} l(r)}{|x|^d (1 + |x|)^{-\alpha} (l(|x|))^2}
\end{equation*}
when $|x| \ge 2 r$, and
\begin{equation*}
 U(x) \approx c(d, \alpha, C) \, \frac{(1 + |x|)^\alpha}{|x|^d (l(|x|))^2}
\end{equation*}
for all $x \ne 0$. Furthermore, \eqref{eq:scaling} holds with $R_\infty = \infty$, hence Theorems~\ref{thm:bhi} and~\ref{thm:martin} are applicable. By Theorem~\ref{thm:halfspace}, if $d \ge 3$, then
\begin{equation*}
 G_{H}(x, y) \approx c(d, \alpha, C) \, \frac{h(\delta_x)}{h(\delta_x + |x - y|)} \, \frac{h(\delta_y)}{h(\delta_y+|x-y|)} \, \frac{(1 + |x - y|)^\alpha}{|x - y|^d (l(|x - y|))^2}
\end{equation*}
for all $x, y \in H$, $x \ne y$, where
\begin{equation*}
 h(s) = \sqrt{\frac{(1 + s)^{\alpha}}{l(s)}} \, .
\end{equation*}
Finally, Theorem~\ref{thm:reg} asserts certain smoothness of harmonic functions, determined by the smoothness of the L\'evy measure $\nu$.
\end{exmp}

\begin{exmp}\label{ex:gs}
Let $X_t$ be a geometric stable process in $\R^d$, that is, a L\'evy processes with characteristic exponent $\Psi(z) = \log(1 + |z|^\alpha)$ for some $\alpha \in (0, 2)$. As remarked above, this is a particular example of a class of processes considered in Example~\ref{ex:slow}. Many properties of $X_t$, including bounds on the potential kernel given in Theorem~\ref{thm:pot}, have been proved in~\cite{MR2238934,MR2240700}. In addition, since $\nu$ is smooth, by Theorem~\ref{thm:reg}, harmonic functions for $X_t$ are smooth (see also Remark~\ref{rem:reg}\ref{rem:reg:sbm}).
\end{exmp}

\begin{exmp}\label{ex:vg}
Taking $\alpha = 2$ in the previous example leads to the variance-gamma process, a L\'evy process with characteristic exponent $\Psi(z) = \log(1 + |z|^2)$. This is no longer a special case of the class of processes studied in Example~\ref{ex:slow} due to exponential decay of $\nu$ at infinity. By Theorems~3.4 and~3.5 in~\cite{MR2240700}, in this case
\begin{equation*}
\begin{aligned}
 K(r) & \approx c(d) (1 + r)^{-2} , & \qquad L(r) & \approx c(d) e^{-r} (1 + r)^{-(d + 1)/2} l(r) .
\end{aligned}
\end{equation*}
Theorems~\ref{thm:ret} and~\ref{thm:pot} lead to bounds that are sharp as long as $|x|$ remains bounded, but deteriorate when $|x| \to \infty$: if $d \ge 3$, then
\begin{equation*}
 \frac{1}{c(d)} \, \frac{r^d (1 + r)^{-2} l(r)}{e^{|x|} |x|^d (1 + |x|)^{-(d + 3)/2} (l(|x|))^2} \le \pr^x(T_r < \infty) \le c(d) \, \frac{r^d (1 + r)^{-2} l(r)}{|x|^d (1 + |x|)^{-2} (l(|x|))^2}
\end{equation*}
when $|x| \ge 2 r$, and
\begin{equation*}
 \frac{1}{c(d)} \, \frac{(1 + |x|)^{(d + 3)/2}}{e^{|x|} |x|^d (l(|x|))^2} \le U(x) \le c(d) \, \frac{(1 + |x|)^2}{|x|^d (l(|x|))^2}
\end{equation*}
for all $x \ne 0$. A sharp two-sided bound for $U(x)$ in this case is provided by Theorem~3.3 in~\cite{MR2238934} and Theorem~3.2 in~\cite{MR2240700}. It turns out that the upper bounds given above describe the behaviour of $U(x)$ and $\pr^x(T_r < \infty)$ correctly.

The density of the L\'evy measure $\nu$ is smooth, and so Theorem~\ref{thm:reg} implies smoothness of harmonic functions for $X_t$ (see also Remark~\ref{rem:reg}\ref{rem:reg:sbm}). Since~\eqref{eq:scaling} holds with any $R_\infty < \infty$, Theorems~\ref{thm:bhi} and~\ref{thm:martin} are applicable. This is, however, insufficient for Theorem~\ref{thm:halfspace}.
\end{exmp}

\begin{exmp}\label{ex:sup}
Let $X_t$ be the compound Poisson process with $\nu(z) = (1 / |B_1|) \ind_{B_1}(z) + (\lambda / |B_R|) \ind_{B_R}(z)$, where $R$ and $\lambda$ are very large, but $\lambda / |B_R|$ is very close to zero. The process $X_t$ is a sum of two independent compound Poisson processes $Y_t$ and $Z_t$, corresponding to the two terms in the definition of $\nu(z)$.
\begin{enumerate}[label=(\alph*)]
\item
Suppose that $R = \lambda^3$, and let $f(x) = \pr^x(X(\tau_{B_2}) \in B_3)$. We will prove that the upper bound of Theorem~\ref{thm:sup} is not sharp for $f$, a harmonic function in~$B_2$.

There are two scenarios of $X_t$ starting at $0$ and remaining in $B_3$ after exiting~$B_2$: in the first one, the first three jumps of $X_t$ are the same as the jumps of $Y_t$, which has probability of $1 / (\lambda + 1)^3$; in the other one, at least one of the first three jumps is the same as the first jump of $Z_t$, and this has length smaller than $4$, which has probability at most $(4 / R)^d$. Thus,
\begin{equation*}
 f(0) \le \frac{1}{(\lambda + 1)^3} + \left(\frac{4}{R}\right)^{\!\! d} \le c(d) \, \frac{1}{\lambda^3} \, .
\end{equation*}
On the other hand, if $3/2 \le |x| < 2$, then $X_t$ can start from $x$ and reach $B_3 \setminus B_2$ in a single jump of $Y_t$, so that
\begin{equation*}
 f(x) \ge \frac{1}{\lambda + 1} \cdot \frac{|B(x, 1) \setminus B_2|}{|B(x, 1)|} \ge \frac{1}{c(d)} \, \frac{1}{\lambda} \, .
\end{equation*}
Furthermore, $K(2) = c(d) (1 + \lambda / |B_R|)$ and $L(2) = \lambda (1 - (2 / R)^d)$, so that $K(2) / (K(2) + L(2)) \ge (c(d))^{-1} / \lambda$. Hence.
\begin{equation*}
 \frac{K(2)}{K(2) + L(2)} \int_{B_2 \setminus B_1} f(z) dz \ge \frac{1}{c(d)} \, \frac{1}{\lambda^2} \, .
\end{equation*}
If follows that as $\lambda \to \infty$, the \emph{local term} in the upper bound of Theorem~\ref{thm:sup} (for $x_0 = 0$, $q = 1$, $r = 2$) is already too large, namely, the ratio
\begin{equation*}
 f(0) : \left(\frac{K(2)}{K(2) + L(2)} \int_{B_2 \setminus B_1} f(z) dz \right)
\end{equation*}
converges to $0$ as $\lambda \to \infty$.
\item Suppose that $d \ge 3$ and $R = \lambda^4$. Let $f(x) = \pr^x(T_2 < \infty)$, where $T_2$ is the hitting time of $\overline{B}_2$. Then $f(x) \le \min(1, (|x|/2)^{2 - d})$. Furthermore, if $3 \le |x| \le 5$ and $r = 2 \lambda^2 + 5$, then $r / R \le 1 / \lambda$ (if, say, $\lambda > 2$). There are two scenarios: either the first jump is drawn from a uniform distribution in $B_1$, or it is uniformly distributed in $B_R$; the probabilities of these scenarios are $1 / (\lambda + 1)$ and $\lambda / (\lambda + 1)$, respectively. Thus,
\begin{equation*}
\begin{aligned}
 f(x) & = \frac{1}{\lambda + 1} \, \frac{1}{|B_1|} \int_{B_1} f(x + z) dz + \frac{\lambda}{\lambda + 1} \, \frac{1}{|B_R|} \int_{B_R} f(x + z) dz \\
 & \le \frac{1}{\lambda + 1} \cdot 1 + \frac{\lambda}{\lambda + 1} \left( \frac{|B_r|}{|B_R|} + \frac{|B_R| - |B_r|}{|B_R|} \, \frac{2^{d - 2}}{(r - |x|)^{d - 2}} \right) \\
 & \le \frac{1}{\lambda} + \frac{r^d}{R^d} + \frac{2^{d - 2}}{(r - 5)^{d - 2}} \le \frac{1}{\lambda} + \frac{1}{\lambda^2} + \frac{1}{\lambda^2} \le \frac{3}{\lambda} \, .
\end{aligned}
\end{equation*}
Repeating the same argument when $|x| = 4$, but using the bound $f(x + z) \le 3 / \lambda$ found above instead of the inequality $f(x + z) \le 1$ for $z \in B_1$, we obtain
\begin{equation*}
\begin{aligned}
 f(x) & \le \frac{1}{\lambda + 1} \cdot \frac{3}{\lambda} + \frac{\lambda}{\lambda + 1} \left( \frac{|B_r|}{|B_R|} + \frac{|B_R| - |B_r|}{|B_R|} \, \frac{2^{d - 2}}{(r - 4)^{d - 2}} \right) \\
 & \le \frac{3}{\lambda^2} + \frac{r^d}{R^d} + \frac{2^{d - 2}}{(r - 4)^{d - 2}} \le \frac{3}{\lambda^2} + \frac{1}{\lambda^2} + \frac{1}{\lambda^2} \le \frac{5}{\lambda^2} \, .
\end{aligned}
\end{equation*}
Hence, the upper bound of Theorem~\ref{thm:ret} (with $r = 2$) is not sharp when $\lambda \to \infty$.
\item In a similar way, the upper bound of Theorem~\ref{thm:pot} is not sharp when $|x| = 2$ and $\lambda \to \infty$; we omit the details.
\end{enumerate}
We remark that although $X_t$ is a compound Poisson process, similar results can be proved for a L\'evy process with non-zero Gaussian component, namely, for the L\'evy process with $\sigma^2 = 1$ and $\nu(z) = (\lambda / |B_R|) \ind_{B_R}(z)$, where, again, $R$ and $\lambda$ are very large, but $\lambda / |B_R|$ is very close to zero.
\end{exmp}

\begin{exmp}\label{ex:hi}
Let $X_t$ be the L\'evy process with $\sigma^2 = 0$ and $\nu(z) = |z|^{-d} \ind_B(x)$ (where $B$ is the unit ball). Note that $\nu$ is not integrable, and so $X_t$ is not a compound Poisson process. The L\'evy process $X_t$ clearly does not satisfy~\eqref{eq:scaling}. We will prove that the scale invariant Harnack inequality does \emph{not} hold as well, and so Theorem~\ref{thm:bhi} cannot be extended to general L\'evy processes, at least in the form given above.

It is enough to consider harmonic functions $f_r(x) = \pr^x(X_{\tau_{B_{2 r}}} \in A)$, where $0 < r < 1/4$ and $A = \{z \in \R^d : z_d > 1\}$. Let $x_r = (0, 0, \ldots, 0, r)$. We will prove that
\begin{equation}\label{eq:lim_exmp}
 \lim_{r \to 0^+} \frac{f_r(-x_r)}{f_r(x_r)} = 0 .
\end{equation}
Denote $H = \{z \in \R^d : z_d > 0\}$. Observe that if $z \in A$, $y \in B_{2 r}$ and $\nu(z - y) > 0$, then $y \in H$ and $|z - y| \ge 1/2$. Thus, by Theorem~\ref{thm:green}, the Ikeda--Watanabe formula~\eqref{eq:iw} and radial monotonicity of $\nu$, we have
\begin{equation*}
\begin{aligned}
 f_r(-x_r) & = \int_{B_{2 r} \cap H} \left(\int_A G_{B_{2 r}}(-x_r, y) \nu(z - y) dz\right) dy \\
 & \le c(d) \, \frac{K(r)}{r^d (K(r) + L(r))^2} \, \nu(1/2) \int_{B_{2 r} \cap H} \left(\int_{A \cap B(y, 1)} dz\right) dy \\
 & \le c(d) \frac{r^{(d+1)/2} K(r)}{(K(r) + L(r))^2}
\end{aligned}
\end{equation*}
(for the last step, observe that $|A \cap B(y, 1)| \le c(d) r^{(d + 1)/2}$ for all $y \in B_{2 r}$). On the other hand, by the lower bound of Lemma~\ref{lem:pb:bound} and $\nu(z) \ge 1$ when $|z| < 1$,
\begin{equation*}
\begin{aligned}
 f_r(x_r) & \ge \int_A P_{B(x_r, r)}(x_r, z) dz \ge \frac{1}{c(d)} \, \frac{1}{K(r) + L(r)} \int_A \nu(z - x_r) dz \\
 & \ge \frac{1}{c(d)} \, \frac{1}{K(r) + L(r)} \,  |A \cap B(x_r, 1)| \ge \frac{1}{c(d)} \, \frac{r^{(d + 1)/2}}{K(r) + L(r)} \, .
\end{aligned}
\end{equation*}
Formula~\eqref{eq:lim_exmp} follows from the above estimates, because when $0 < r < 1$, we have $K(r) = c(d)$ and $L(r) = c(d) \log(1 / r)$.
\end{exmp}

%
%

\appendix

%
%

\section{O-regularly varying functions}
\label{sec:app:orv}

Our results rely on the comparability condition~\eqref{eq:scaling} on the density of the L\'evy measure, known as \emph{$O$-regular variation} or \emph{lower scaling condition}. For reader's convenience we give a short survey of these notion. For further information, see Section~2 in~\cite{MR898871}.

A positive-valued function $\varphi$ is said to be \emph{$O$-regularly varying} at infinity if
\begin{equation}\label{eq:appB:sc}
 A \, (r_2 / r_1)^a \le \varphi(r_2) / \varphi(r_1) \le B \, (r_2 / r_1)^b
\end{equation}
when $R_0 < r_1 < r_2$, for some constants $A, B, R_0 > 0$, $a, b \in \R$. The supremum $\alpha$ of the numbers $a$ for which the lower bound is satisfied for some $A, R_0 > 0$ is called the \emph{lower Matuszewska index} (or simply the \emph{lower index}) of $\varphi$ at infinity, while the analogous infimum $\beta$ of the numbers $b$ --- the \emph{upper Matuszewska index} of $\varphi$ at infinity. There are at least four equivalent definitions of Matuszewska indices; for a thorough discussion, see~\cite{MR0466438}. The definition based on inequality~\eqref{eq:appB:sc} is particularly useful for our needs.

The notions of $O$-regular variation at zero and Matuszewska indices at zero are defined in a similar way: $\varphi$ is $O$-regularly varying at zero if condition~\eqref{eq:appB:sc} is satisfied when $0 < r_1 < r_2 < R_0$, for some constants $A, B, R_0 > 0$, $a, b \in \R$. Thus, $\varphi$ is $O$-regularly varying both at zero and at infinity if and only if condition~\eqref{eq:appB:sc} holds when $0 < r_1 < r_2$. Noteworthy, $\varphi$ is $O$-regularly varying at zero with lower index $\alpha$ and upper index $\beta$ if and only if the function $\varphi(1/s)$ is $O$-regularly varying at infinity with lower index $-\beta$ and upper index $-\alpha$.

Throughout this section we use $f \approx C g$ as a short-hand notation for $C^{-1} g \le f \le C g$.

The following well-known result was proved in~\cite{MR0466438} (see also~\cite{MR898871}): the first two statements are consequences of Theorem~3 and Lemmas~4 and~4' in~\cite{MR0466438}, while the other two follow from the former two by considering the function $\varphi(1 / r)$.

\begin{prop}[see~\cite{MR0466438}]\label{prop:orv}
Let $\varphi$ be a positive function on $(0, \infty)$.
\begin{enumerate}[label={\rm{(\alph*)}}]
\item\label{prop:orv:a}
If $\varphi$ is $O$-regularly varying at infinity with upper index $\beta$ and $s > \beta$, then there are $C, R > 0$ such that for all $r > R$,
\begin{equation*}
 \int_r^\infty t^{-s} \varphi(t) \frac{dt}{t} \approx C \, r^{-s} \varphi(r) .
\end{equation*}
Conversely, if the latter condition holds for some $C, R > 0$, $s \in \R$, then $\varphi$ is $O$-regularly varying at infinity with upper index less than $s$.
\item\label{prop:orv:b}
If $\varphi$ is $O$-regularly varying at infinity with lower index $\alpha$ and $s < \alpha$, then there are $C, R > 0$ such that for all $r \ge 2 R$,
\begin{equation*}
 \int_R^r t^{-s} \varphi(t) \frac{dt}{t} \approx C \, r^{-s} \varphi(r) .
\end{equation*}
Conversely, if the latter condition holds for some $C, R > 0$, $s \in \R$, then $\varphi$ is $O$-regularly varying at infinity with lower index greater than $s$.
\item\label{prop:orv:c}
If $\varphi$ is $O$-regularly varying at zero with lower index $\alpha$ and $s < \alpha$, then there are $C, R > 0$ such that whenever $0 < r < R$,
\begin{equation*}
 \int_0^r t^{-s} \varphi(t) \frac{dt}{t} \approx C \, r^{-s} \varphi(r) .
\end{equation*}
Conversely, if the latter condition holds for some $C, R > 0$, $s \in \R$, then $\varphi$ is $O$-regularly varying at zero with lower index greater than $s$.
\item\label{prop:orv:d}
If $\varphi$ is $O$-regularly varying at zero with upper index $\beta$ and $s > \beta$, then there are $C, R > 0$ such that whenever $0 < r \le R/2$,
\begin{equation*}
 \int_r^R t^{-s} \varphi(t) \frac{dt}{t} \approx C \, r^{-s} \varphi(r) .
\end{equation*}
Conversely, if the latter condition holds for some $C, R > 0$, $s \in \R$, then $\varphi$ is $O$-regularly varying at zero with upper index less than $s$.
\end{enumerate}
\end{prop} 

In Remark~\ref{rem:orv}\ref{rem:orv:karamata} we referred to the following direct corollary of Proposition~\ref{prop:orv}\ref{prop:orv:b} and~\ref{prop:orv:c}.

\begin{cor}[see~\cite{MR0466438}]\label{cor:orv:unimodal}
Let $\varphi$ be a non-increasing, positive-valued function on $(0, \infty)$. Let $R > 0$ and $s \le 0$. If $\varphi$ has lower indices $\alpha_0$, $\alpha_\infty$ at zero and at infinity, and $\min(\alpha_0, \alpha_\infty) > s$, then there is $C > 0$ such that for all $r > 0$,
\begin{equation}\label{eq:orv:unimodal}
 \int_0^r t^{-s} \varphi(t) \frac{dt}{t} \approx C \, r^{-s} \varphi(r) .
\end{equation}
Conversely, if the latter condition holds for some $C > 0$, then $\varphi$ has lower indices $\alpha_0$, $\alpha_\infty$ at zero and at infinity, and $\min(\alpha_0, \alpha_\infty) > s$.
\end{cor}

Proposition~\ref{prop:orv} is complemented by the following version of monotone density theorem, whose second part is Proposition~2.10.3 in~\cite{MR898871}. The proof of the first statement is very similar (and we omit the details), and the last two parts follow from the first two by considering the function $\varphi(1 / r)$.

\begin{prop}[see~\cite{MR898871}]\label{prop:orvmd}
Let $\varphi$ be positive and locally integrable function on $(0, \infty)$.
\begin{enumerate}[label={\rm{(\alph*)}}]
\item\label{prop:orvmd:a}
Suppose that $r^{-s} \varphi(r)$ is integrable at infinity, and define
\begin{equation*}
 \Phi(r) = \int_r^\infty t^{-s} \varphi(t) \, \frac{dt}{t} \, .
\end{equation*}
If $\varphi$ has finite lower or upper index at infinity, and $\Phi$ is $O$-regularly varying at infinity with negative upper index, then there are $C, R > 0$ such that $\Phi(r) \approx C r^{-s} \varphi(r)$ for all $r > R$.
\item\label{prop:orvmd:b}
Define
\begin{equation*}
 \Phi(r) = \int_1^r t^{-s} \varphi(t) \, \frac{dt}{t} \, .
\end{equation*}
If $\varphi$ has finite lower or upper index at infinity, and $\Phi$ is $O$-regularly varying at infinity with positive lower index, then there are $C, R > 0$ such that $\Phi(r) \approx C r^{-s} \varphi(r)$ for all $r > R$.
\item\label{prop:orvmd:c}
Suppose that $r^{-s} \varphi(r)$ is integrable near zero, and define
\begin{equation*}
 \Phi(r) = \int_0^r t^{-s} \varphi(t) \, \frac{dt}{t} \, .
\end{equation*}
If $\varphi$ has finite lower or upper index at zero, and $\Phi$ is $O$-regularly varying at zero with positive lower index, then there are $C, R > 0$ such that $\Phi(r) \approx C r^{-s} \varphi(r)$ whenever $0 < r < R$.
\item\label{prop:orvmd:d}
Define
\begin{equation*}
 \Phi(r) = \int_r^1 t^{-s} \varphi(t) \, \frac{dt}{t} \, .
\end{equation*}
If $\varphi$ has finite lower or upper index at zero, and $\Phi$ is $O$-regularly varying at zero with negative upper index, then there are $C, R > 0$ such that $\Phi(r) \approx C r^{-s} \varphi(r)$ whenever $0 < r < R$.
\end{enumerate}
\end{prop}

Remark~\ref{rem:orv}\ref{rem:orv:fourier} mentioned the relation between $O$-regular variation of the density of the L\'evy measure $\nu$ and the characteristic exponent $\Psi$, which follows rather easily from Propositions~\ref{prop:orv} and~\ref{prop:orvmd}. For completeness, we provide a short proof.

Note that for isotropic unimodal L\'evy processes, by~\eqref{eq:lk}, $\Psi(z)$ is comparable with $K(1/r) + L(1/r)$, which immediately implies that $\Psi$ is $O$-regularly varying.

\begin{prop}\label{prop:lk}
Let $\Psi(|z|)$ be the characteristic exponent of an isotropic unimodal L\'evy process with Gaussian coefficient $\sigma^2$ and density of the L\'evy measure $\nu(|z|)$.
\begin{enumerate}[label={\rm{(\alph*)}}]
\item\label{prop:lk:a}
The function $\Psi$ has upper index $\beta \in [0, 2)$ at zero if and only if $L$ has lower index $-\beta \in (-2, 0]$ at infinity.
\item\label{prop:lk:b}
The function $\Psi$ has lower index $\alpha \in (0, 2]$ at zero if and only if $K$ has upper index $-\alpha \in [-2, 0)$ at infinity.
\item\label{prop:lk:c}
The function $\Psi$ has upper index $\beta \in [0, 2)$ at infinity if and only if $\sigma^2 = 0$ and $L$ has lower index $-\beta \in (-2, 0]$ at zero.
\item\label{prop:lk:d}
The function $\Psi$ has lower index $\alpha \in (0, 2]$ at infinity if and only if $K$ has upper index $-\alpha \in [-2, 0)$ at zero.
\end{enumerate}
In case~\ref{prop:lk:a}, the functions $L(r)$, $K(r) + L(r)$ and $\Psi(1 / r)$ are comparable for $r > 1$. Similar statements hold true for the other cases.
\end{prop}

\begin{proof}
Suppose that the former condition of~\ref{prop:lk:a} holds true. Since $\Psi(1/r)$ is comparable with $K(r) + L(r)$, the lower index of $K(r) + L(r)$ at infinity is $-\beta$, that is, the lower index of $r^2 (K(r) + L(r))$ at infinity is $2 - \beta$. Observe that
\begin{equation}\label{eq:lk:1}
\begin{aligned}
 r^2 (K(r) + L(r)) & = \sigma^2 d + \int_{\R^d} \min(|z|^2, r^2) \nu(z) dz \\
 & = \sigma^2 d + \int_{\R^d} \left( \int_0^r 2 t \ind_{\R^d \setminus B_t}(z) dt \right) \nu(z) dz = \sigma^2 d + \int_0^r 2 t^2 L(t) \, \frac{dt}{t} .
\end{aligned}
\end{equation}
Since $L(r)$ has finite upper index at infinity (it is decreasing), Proposition~\ref{prop:orvmd}\ref{prop:orvmd:b} implies that $r^2 (K(r) + L(r))$ is comparable with $2 r^2 L(r)$ for $r > 1$, and so $L(r)$ is $O$-regularly varying at infinity with lower index $-\beta$. Following the above steps in the opposite direction (which involves Proposition~\ref{prop:orv}\ref{prop:orv:a}) shows that the latter condition of~\ref{prop:lk:a} implies the former one, which completes the proof of~\ref{prop:lk:a}.

The proof of~\ref{prop:lk:b} is similar: if the former condition of~\ref{prop:lk:b} is satisfied, then the upper index of $K(r) + L(r)$ at infinity us $-\beta$. However,
\begin{equation}\label{eq:lk:2}
\begin{aligned}
 K(r) + L(r) & = \frac{\sigma^2 d}{r^2} + \int_{\R^d} \min \! \left( \frac{|z|^2}{r^2} \, , 1 \right) \nu(z) dz \\
 & = \int_r^\infty \frac{2 \sigma^2 d}{s^3} \, ds + \int_{\R^d} \left( \int_r^\infty \frac{2 |z|^2}{s^3} \, \ind_{B_s}(z) ds \right) \nu(z) dz =  \int_r^\infty 2 K(s) \, \frac{ds}{s} .
\end{aligned}
\end{equation}
Since $2 K(s)$ has finite upper index ($s^2 K(s)$ is increasing), by Proposition~\ref{prop:orvmd}\ref{prop:orvmd:a}, $K(r) + L(r)$ is comparable with $2 K(r)$ for $r > 1$, and hence $K(r)$ is $O$-regularly varying at infinity with upper index $-\beta$, as desired. Again, the converse is proved by reversing the steps in the above argument and using Proposition~\ref{prop:orv}\ref{prop:orv:a}.

The proofs of~\ref{prop:lk:c} and~\ref{prop:lk:d} are very similar. If the former condition of~\ref{prop:lk:c} holds, then $r^2 (K(r) + L(r))$ has lower index $2 - \beta$ at zero, and by~\eqref{eq:lk:1} and Proposition~\ref{prop:orvmd}\ref{prop:orvmd:c}, $\sigma^2 = 0$ and $2 r^2 L(r)$ has lower index $2 - \beta$ at zero. The converse involves Proposition~\ref{prop:orv}\ref{prop:orv:c}. In the same way, the former condition of~\ref{prop:lk:d} implies that $K(r) + L(r)$ has upper index $-\beta$ at zero, and by~\eqref{eq:lk:2} and Proposition~\ref{prop:orvmd}\ref{prop:orvmd:d}, $2 K(r)$ has upper index $-\beta$ at zero; for the converse, use Proposition~\ref{prop:orv}\ref{prop:orv:d}.
\end{proof}

\begin{remark}
The above result extends automatically to general L\'evy processes if the profile of the characteristic exponent $\Psi(r)$ is replaced by $\Psi^*(r) = \sup\{\re \Psi(z) : z \in B_r\}$, see Lemma~4 in~\cite{TG}.
\end{remark}

The function $L(r)$ may have lower index greater than $-2$ (as in Proposition~\ref{prop:lk}\ref{prop:lk:a}) even if the lower index of $\nu(r)$ at infinity is less than $-2$ (possibly even $-\infty$). This is, however, not possible if we impose an additional Tauberian-type condition. On the other hand, it is easy to see that if $\nu(r)$ has lower index at infinity greater than $-d - 2$, then $L$ has lower index at infinity greater than $-2$. Since these observations were mentioned in Remark~\ref{rem:orv}\ref{rem:orv:fourier}, we collect them in the following statement.

\begin{prop}\label{prop:lknu}
Let $\Psi(|z|)$ be the characteristic exponent of an isotropic unimodal L\'evy process with  Gaussian coefficient $\sigma^2$ and density of the L\'evy measure $\nu(|z|)$.
\begin{enumerate}[label={\rm{(\alph*)}}]
\item\label{prop:lknu:a}
If $\nu$ has lower index $-d - \beta$ at infinity for some $\beta \in [0, 2)$, then $\Psi$ has upper index at most $\beta$ at zero.
\item\label{prop:lknu:b}
If $\Psi$ has upper index $\beta \in [0, 2)$ at zero, and in addition any of the following conditions is satisfied:
\begin{itemize}
\item $\Psi$ has positive lower index at zero;
\item $\Psi$ belongs to de~Hahn's class at $0$ (and thus $\beta = 0$; see~\cite{MR898871,GRT} for definitions),
\end{itemize}
then $\nu$ has lower index $-d - \beta$ at infinity.
\end{enumerate}
\end{prop}

\begin{proof}
For part~\ref{prop:lknu:a}, observe that, for any $\varepsilon > 0$, there is a constant $C > 0$ such that $\nu(t) \ge C (t / s)^{-d - \beta - \varepsilon} \nu(s)$ when $1 < s < t$, and hence, when $1 < r_1 < r_2$,
\begin{equation*}
\begin{aligned}
 L(r_2) & = c(d) \int_{r_2}^\infty t^{d-1} \nu(t) dt = c(d) r_2^d \int_1^\infty s^{d-1} \nu(r_2 s) ds \\
 & \ge C c(d) r_2^d \int_1^\infty s^{d-1} (r_2 / r_1)^{-d - \beta - \varepsilon} \nu(r_1 s) ds \\
 & = C c(d) (r_2 / r_1)^{-\beta - \varepsilon} \, r_1^d \int_1^\infty s^{d-1} \nu(r_1 s) ds = C (r_2 / r_1)^{-\beta - \varepsilon} L(r_1) ,
\end{aligned}
\end{equation*}
and so $L$ has lower index at least $-\beta$ at infinity. The desired result follows from Proposition~\ref{prop:lk}~\ref{prop:lk:a}.

Suppose now that $\Psi$ has lower index $\alpha \in (0, 2]$ and upper index $\beta \in [0, 2)$ at zero. By Proposition~\ref{prop:lk}, $L(r)$ is comparable with $K(r)$ for $r > 1$, and both functions have upper index $-\alpha$ and lower index $-\beta$ at infinity. By Proposition~\ref{prop:orvmd}\ref{prop:orvmd:a}, $L(r)$ is comparable with $r^d \nu(r)$ for $r > 1$, and so the lower index of $\nu(r)$ at infinity is $-d - \beta$, as desired.

If $\Psi$ is in de~Hahn's class at $0$, then the assertion follows from Theorem~3.6 in~\cite{GRT}.
\end{proof}

%
%

\section{Integral average of monotone kernels}
\label{sec:app:sup}

The following result is likely well known, but the authors failed to find a reference in the literature.

\begin{lem}
\label{lem:avg}
For $r \in [0, 1]$ let $\pi_r(ds)$ be a non-zero measure on $[r, \infty)$ which may contain an atom at $r$ and which has a non-increasing density function $\pi_r(s)$ on $(r, \infty)$ (with respect to the Lebesgue measure). Suppose furthermore that $\pi_r$ depends continuously (with respect to the vague topology) on $r \in [0, 1]$. Then there is a measure $\mu$ on $[0, 1]$ such that the corresponding integral average of $\pi_r$:
\begin{equation*}
 \bar{\pi}(E) = \int_{[0, 1]} \pi_r(E) \mu(dr)
\end{equation*}
has a density function (with respect to the Lebesgue measure) $\bar{\pi}(s)$ which is equal to $1$ on $[0, 1]$, and it is a non-increasing function on $[1, \infty)$.
\end{lem}

\begin{proof}
Fix $n \ge 1$ and denote $I_j = [j/n, (j+1)/n)$. Define $\alpha_j$ for $j = 0, 1, \ldots, n - 1$ recursively by the formula
\begin{equation*}
 (\alpha_0 \pi_0 + \alpha_1 \pi_{1/n} + \ldots + \alpha_j \pi_{j/n})(I_j) = 1/n .
\end{equation*}
This properly defines $\alpha_j$, because $\pi_{j/n}(I_j) > 0$. We prove by induction that $\alpha_j \ge 0$. For $j = 0$ this is clear. Furthermore, if for some $j = 1, 2, \ldots, n - 1$ we have $\alpha_0, \alpha_1, \ldots, \alpha_{j-1} \ge 0$, then (by the assumptions on $\pi_r$)
\begin{equation*}
 (\alpha_0 \pi_0 + \alpha_1 \pi_{1/n} + \ldots + \alpha_{j-1} \pi_{(j-1)/n})(I_j) 
 \le (\alpha_0 \pi_0 + \alpha_1 \pi_{1/n} + \ldots + \alpha_{j-1} \pi_{(j-1)/n})(I_{j-1}) = 1/n ,
\end{equation*}
which implies that $\alpha_j \ge 0$, as desired. We define
\begin{equation*}
 \mu = \alpha_0 \delta_0 + \alpha_1 \delta_{1/n} + \ldots + \alpha_{n-1} \delta_{(n-1)/n} ,
\end{equation*}
and
\begin{equation*}
 \bar{\pi}(E) = \int_{[0,1]} \pi_r(E) \mu(dr) = \alpha_0 \pi_0(E) + \alpha_1 \pi_{1/n}(E) + \ldots + \alpha_{n-1} \pi_{(n-1)/n}(E) .
\end{equation*}
Then $\bar{\pi}([j/n, (j+1)/n)) = 1/n$ for every $j = 0, 1, \ldots, n - 1$, and $\bar{\pi}$ has a non-increasing density function (with respect to the Lebesgue measure) on $(1 - 1/n, \infty)$. Since $\bar{\pi}([1 - 1/n, 1)) = 1/n$, the density function of $\bar{\pi}$ is not greater than $1$ on $[1, \infty)$. The measures $\mu$ and $\bar{\pi}$ are therefore reasonable approximations to the measures sought in the lemma.

Consider $n = 2^m$ for $m = 1, 2, \ldots$ (which corresponds to nested dyadic partitions of $[0, 1]$), and denote the objects constructed above by $\mu^{(m)}$ and $\bar{\pi}^{(m)}$. Observe that for every dyadic interval $E \sub [0, 1]$, $\bar{\pi}^{(m)}(E) = |E|$ for $m$ large enough. Therefore, $\bar{\pi}^{(m)}$ restricted to $[0, 1]$ converges vaguely to the Lebesgue measure on $[0, 1]$. Furthermore, there is a subsequence of $\mu^{(m)}$ which is vaguely convergent to a measure $\mu$ on $[0, 1]$. For simplicity, we denote this subsequence again by $\mu^{(m)}$, and we let
\begin{equation*}
 \bar{\pi}(E) = \int_{[0,1]} \pi_r(E) \mu(dr) .
\end{equation*}
For any continuous, compactly supported function $f$ on $[0, \infty)$, $\pi_r(f)$ (which is a short-hand notation for $\int_{[0, \infty)} f(s) \pi_r(ds)$) depends continuously on $r$. Hence
\begin{equation*}
 \lim_{m \to \infty} \bar{\pi}^{(m)}(f) = \lim_{m \to \infty} \int_{[0,1]} \pi_r(f) \mu^{(m)}(dr) = \int_{[0,1]} \pi_r(f) \mu(dr) = \bar{\pi}(f) .
\end{equation*}
Therefore, $\bar{\pi}^{(m)}$ is vaguely convergent to $\bar{\pi}$. In particular, $\bar{\pi}$ restricted to $[0, 1]$ is the Lebesgue measure on $[0, 1]$. Furthermore, all measures $\bar{\pi}^{(m)}$ have density functions (with respect to the Lebesgue measure) on $(1, \infty)$ which are non-increasing and bounded from above by $1$. Therefore, also $\bar{\pi}$ has this property.
\end{proof}

By a linear change of variables, we immediately obtain the following extension.

\begin{cor}
\label{cor:avg}
Let $q \in [0, 1)$. For $r \in [q, 1]$ let $\pi_r(ds)$ be a non-zero measure on $[r, \infty)$ which may contain an atom at $r$ and which has a non-increasing density function $\pi_r(s)$ on $(r, \infty)$ (with respect to the Lebesgue measure). Suppose furthermore that $\pi_r$ depends continuously (with respect to the vague topology) on $r \in [q, 1]$. Then there is a measure $\mu$ on $[q, 1]$ such that the corresponding integral average of $\pi_r$:
\begin{equation*}
 \bar{\pi}(E) = \int_{[q, 1]} \pi_r(E) \mu(dr)
\end{equation*}
has a density function (with respect to the Lebesgue measure) $\bar{\pi}(s)$ which is equal to $0$ on $[0, q)$, equal to $1$ on $[q, 1]$, and it is a non-increasing function on $[1, \infty)$.
\end{cor}

We say that $P(dz)$ is an isotropic unimodal measure on $\R^d \setminus B(0, r)$ if it is invariant under rotations, and the radial part $\pi(dt)$, defined by $P(B(0, s)) = \int_{[0, s)} \omega_d t^{d-1} \pi(dt)$, is zero on $[0, r)$, may contain an atom at $r$, and has a non-increasing density function (with respect to the Lebesgue measure) on $(r, \infty)$ (this is the same definition as the one used in Corollary~\ref{cor:pb:unimodal}).

\begin{lem}\label{lem:reg}
Let $q \in [0, 1)$. For $r \in [q, 1]$ let $P_r(dz)$ be an isotropic unimodal probability measure on $\R^d \setminus B(0, r)$ which depends continuously (with respect to the vague topology) on $r \in [0, 1]$. Then there is a probability measure $\mu$ on $[q, 1]$ such that the corresponding integral average of $P_r$:
\begin{equation*}
 \bar{P}(E) = \int_{[q, 1]} P_r(E) \mu(dr)
\end{equation*}
has a radial density function (with respect to the Lebesgue measure) $\bar{P}(z)$ with the following properties: $\bar{P}(z) = 0$ in $B(0, q)$, $\bar{P}(z) = c$ on $B(0, 1) \setminus B(0, q)$ for some constant $c > 0$, and $\bar{P}(z) \le c$ on $\R^d \setminus B(0, 1)$.
\end{lem}

Note that from the definition it follows that $\bar{P}(z)$ is a non-increasing function of $|z|$ when $|z| > 1$.

\begin{proof}
For $r \in [q, 1]$, let $\pi_r$ be the radial part of $P_r$, as defined before the statement of the lemma. By Corollary~\ref{cor:avg}, there is a measure $\tilde{\mu}$ on $[q, 1]$ such that the measure
\begin{equation*}
 \bar{\pi}(E) = \int_{[q, 1]} \pi_r(E) \tilde{\mu}(dr)
\end{equation*}
has a density function $\bar{\pi}(s)$ (with respect to the Lebesgue measure) which is equal to $0$ on $[0, q)$, equal to $1$ on $[q, 1]$, and is a non-increasing function (bounded from above by $1$) on $[1, \infty)$. We define $c = (\tilde{\mu}([q, 1]))^{-1}$ and
\begin{equation*}
 \bar{P}(E) = c \int_{[q, 1]} P_r(E) \tilde{\mu}(dr) .
\end{equation*}
Then $P$ is an isotropic measure, and
\begin{multline*}
 \bar{P}(B(0, s)) = c \int_{[q, 1]} P_r(B(0, s)) \tilde{\mu}(dr) = c \int_{[q, 1]} \int_{[0, s)} t^{d-1} \pi_r(dt) \tilde{\mu}(dr) \\
 = c \int_{[0, s)} t^{d-1} \bar{\pi}(dt) = c \int_{[0, s)} t^{d-1} \bar{\pi}(t) dt = c \int_{B(0, s)} \bar{\pi}(|z|) dz .
\end{multline*}
Therefore, $\bar{P}(z) = c \, \bar{\pi}(|z|)$ is the density function of the measure $\bar{P}$.
\end{proof}

\subsection*{Acknowledgments}

We thank the anonymous referee for numerous helpful comments.

%
%

%
%

\end{document}